\documentclass[twoside,11pt,a4wide,reqno]{amsart}

\usepackage{amsmath,amsfonts,amsthm,amssymb,amsbsy,upref,color,graphicx,amscd,mathrsfs,enumerate}
\usepackage[colorlinks=true,linkcolor=blue,citecolor=blue,
filecolor=blue,urlcolor=blue,pageanchor=true,plainpages=false,
linktocpage]{hyperref}
\usepackage[active]{srcltx}
\usepackage[latin1]{inputenc}
\usepackage{bbm}
\usepackage{pdfsync}
\usepackage{mathrsfs}
\usepackage{todonotes}
\usepackage{eucal}

\DeclareMathOperator{\argmin}{arg\,min}

\newcommand{\restr}[1]{\lower3pt\hbox{$|_{#1}$}} 

\newcommand{\cconv}{\overline{\mathrm{co}}}

\def\eps{{\varepsilon}}
\def\N{\mathbb{N}}

\def\la{\lambda}
\def\R{\mathbb{R}}

\def\HH{\mathscr{H}}

\def\CC{\mathcal{C}}

\def\E{\mathcal{E}}
\def\calQ{\mathcal{Q}}

\def\d{\mathsf{d}}
\def\dom{{\mathbb V}}
\newcommand{\be}{\begin{equation}}
\newcommand{\ee}{\end{equation}}

\numberwithin{equation}{section}
\theoremstyle{plain}
\newtheorem{theorem}{Theorem}[section]
\newtheorem{lemma}[theorem]{Lemma}
\newtheorem{corollary}[theorem]{Corollary}
\newtheorem{proposition}[theorem]{Proposition}
\newtheorem{definition}[theorem]{Definition}
\newtheorem{remark}[theorem]{Remark}

\newcommand{\mean}[1]{\,-\hskip-1.08em\int_{#1}} 

\makeatletter
\newcommand{\mylabel}[2]{#2\def\@currentlabel{#2}\label{#1}}
\makeatother

\newcommand{\mm}{{\mathfrak m}}
\renewcommand{\d}{\mathrm d}
\newcommand{\sfD}{{\mathsf D}}
\newcommand{\sfQ}{{\mathsf Q}}
\newcommand{\rmL}{{\mathrm L}}
\newcommand{\rmR}{{\mathrm R}}

\newcommand{\rmU}{{\mathrm U}}
\newcommand{\brmU}{\boldsymbol{\mathrm{U}}}
\newcommand{\functional}{\mathscr F}
\newcommand{\FF}{\mathscr F}
\newcommand{\tFF}{\mathscr G}
\newcommand{\GG}{\mathscr K}
\newcommand{\Vmin}{V_{\rm min}}
\newcommand{\llambda}{{\boldsymbol \lambda}}

\newcommand{\ssigma}{{\boldsymbol \sigma}}
\newcommand{\mmu}{{\boldsymbol \mu}}
\newcommand{\uU}{{\boldsymbol U}}
\newcommand{\uu}{{\boldsymbol u}}
\newcommand{\ww}{{\boldsymbol w}}
\newcommand{\Hilb}{{\mathbb H}}
\newcommand{\V}{\dom}
\newcommand{\K}{{\mathbb K}}
\newcommand{\Kc}[1]{\K[#1]}
\newcommand{\J}{J}
\newcommand{\rmB}{{\mathrm B}}
\newcommand{\rmC}{{\mathrm C}}
\newcommand{\supp}{\mathrm{supp}}
\newcommand{\weakto}{\rightharpoonup}
\newcommand{\subH}{\mathrm S}
\newcommand{\Int}{\operatorname{Int}}
\DeclareMathOperator{\Ort}{Ort}
\DeclareMathOperator{\Span}{Span}

\title{$L^2$-Gradient Flows of Spectral Functionals}

\author{Dario Mazzoleni}
\address{Dipartimento di Matematica\\
University of Pavia\\
Via Ferrata 5\\
27100 Pavia, Italy}
\email{dario.mazzoleni@unipv.it}
\author{Giuseppe Savar\'e}
\address{Department of Decision Sciences and BIDSA, Bocconi University\\
Via Roentgen 1\\
20136 Milan, Italy}
\email{giuseppe.savare@unibocconi.it}
\thanks{G.S.
  gratefully acknowledges the support of the Institute of Advanced
  Study of the Technical University of Munich and of 
  IMATI-CNR, Pavia. G.S. has been supported by the MIUR-PRIN 2017
  project \emph{Gradient flows, Optimal Transport and Metric Measure
    Structures.}
  The authors are members of the 
  Gruppo Nazionale per l'Analisi Matematica, la Probabilit\`a
				e le loro Applicazioni (GNAMPA) of the 
				Istituto Nazionale di Alta Matematica (INdAM).
				D.M.~has been partially supported by the INdAM-GNAMPA 2019 project ``Ottimizzazione spettrale non lineare'' and MIUR-PRIN 2020 \emph{Mathematics for Industry 4.0}}


\begin{document}

\begin{abstract}
  We study the $L^2$-gradient flow
  of functionals $\FF$ depending on 
the eigenvalues of Schr\"odinger potentials $V$ for a wide class of
differential operators associated with closed,
symmetric, and coercive bilinear forms,
including the case of all the Dirichlet forms
(such as for second order elliptic operators in Euclidean domains or Riemannian manifolds).

We suppose that
$\FF$ arises as the sum of a 
$(-\theta)$-convex functional $\GG$ with proper domain $\K\subset L^2$, forcing the admissible potentials to stay above a constant $V_{\rm min}$, and a term $\HH(V)=\varphi(\lambda_1(V),\cdots,\lambda_\J(V))$
which depends on the first $\J$ eigenvalues associated with $V$ through a  $\rmC^1$ function $\varphi$.

Even though $\HH$ is not a smooth perturbation of a convex functional (and it
is in fact concave in simple important cases as the sum of the first $\J$
eigenvalues) and we do not assume any compactness of the sublevels of
$\GG$, we prove the convergence of the Minimizing Movement method
to a solution $V\in
H^1(0,T;L^2)$
of the differential inclusion $V'(t)\in -\partial_L^-\FF(V(t))$, which
under suitable compatibility conditions on $\varphi$ can be
written as
\[
  V'(t)+\sum_{i=1}^\J\partial_i\varphi(\la_1(V(t)),\dots, \la_\J(V(t)))u_i^2(t)\in -\partial_F^-\GG(V(t))
\]
where $(u_1(t),\dots, u_\J(t))$ is an orthonormal system of
eigenfunctions
associated with the eigenvalues $(\la_1(V(t)),
,\dots,\la_\J(V(t)))$ and $\partial^-_L$ (resp. $\partial^-_F$) denotes the limiting (resp. Fr\'echet) sub\-differential.
\end{abstract}

\maketitle
{\centering\footnotesize
  \it Dedicated to J.L.~Vazquez in occasion of his 75th birthday \par}

\section{Introduction}
Optimization problems for eigenvalues of elliptic operators have been
a subject of great interest in the last few years, due both to the
possible applications and to the
challenging mathematical questions
arising from these topics.

In particular, 
shape optimization problems for the eigenvalues of the Dirichlet Laplacian have been deeply investigated and
many results concerning existence of optimal shapes in suitable
admissible classes of domains, together with regularity results, have
been proved, see~\cite{henrot2,hpnew} for an overview.

A point of view that has not completely been understood yet for this
class of problems is an evolutionary approach through a gradient 
flow of shapes associated with a functional depending on the
eigenvalues. One of the main issues 
is the choice of the natural
metric driving the evolution and of a corresponding topology well
adapted to shape optimization problems.
In the case of stationary variational problems, 
the best approach in order to prove existence results (see~\cite{bdm})
is to relax the problem in the class of capacitary measures,
i.e.~Borel measures that vanish on sets of zero capacity, where the
$\gamma$-convergence provides a compact topology
sufficiently strong to guarantee the continuity of the eigenvalues of the Dirichlet Laplacian.
In the framework of capacitary measures, a first gradient flow evolution for this problem was proposed by Bucur, Buttazzo and Stefanelli in~\cite{bbs}.
They prove existence of (generalized) Minimizing Movements for a large class of functionals, but they do not characterize explicitly the gradient flow equation. 
A very interesting observation from their work is that, even in cases in which the evolution starts from a ``nice'' shape, then the relaxation in the capacitary measures can actually happen. 
{We also quote the approach of
  \cite{Dogan-Morin-Nochetto-Verani} in shape optimization problems.
}

\paragraph{\bf{Eigenvalue problems associated with Schr\"odinger potentials}}
In the present paper we propose a different approach, and we focus on
the evolution, driven by the $L^2$-metric, of a special class of
capacitary measures, that is, those absolutely continuous with respect
to a given reference measure (such as the Lebesgue measure of $\R^d$).
Even though in the strong $L^2$-framework the driving functionals are not
smooth nor convex and their sublevels are not compact, we are still able to prove that the
Minimizing Movements solve a natural differential inclusion.
Our approach is sufficiently strong to deal with eigenvalues of a wide class of
operators,
not only those of the Dirichlet Laplacian, avoiding the relaxation phenomenon.

In fact, we will 
address the problem in the general setting of a (weakly) coercive, symmetric bilinear form
$\E:\dom\times \dom\to\R$ on a Hilbert space 
$\V$ densely and compactly embedded in
$\Hilb=L^2(\sfD,\mm)$ for a finite measure space $(\sfD,\mm)$.
Since $\E$ is a nonnegative quadratic form we have
\begin{equation}
  \label{eq:1}
  \alpha:=\min_{u\in \V\setminus\{0\}}\frac{\E(u,u)}{\int_\sfD
    u^2\,\d\mm}
  \ge 0.
\end{equation}
We consider a convex subset $\K$ of $L^2(\sfD,\mm)$
whose elements $V$ satisfy the uniform lower bound $V(x)\ge V_{\rm
  min}$ for a fixed constant $\Vmin$.
Given a
(Schr\"odinger) potential
$V\in \K$
(that we can identify with the capacitary measure
$\mu=V \mm$ absolutely continuous with respect to
the reference measure $\mm$),
we can introduce the symmetric bilinear form
\begin{equation}
  \label{eq:52}
  \E_V(u,v):=\E(u,v)+\int_\sfD Vuv\,\d\mm\quad
  D(\E_V):=\Big\{u\in \V: \int_\sfD (V_+)u^2\,\d\mm<+\infty\Big\},
\end{equation}
and we say that $\la$ is an eigenvalue associated with
$V$ if there exists a nonzero eigenfunction $u\in D(\E_V)$
solving 
\begin{equation}
  \label{eq:49}
  \E(u,w)+\int_\sfD V\,uw\,\d\mm=\la\int_\sfD uw\,\d\mm
  \quad\text{for every }w\in D(\E_V).
\end{equation}
When $D(\E_V)=\V$ \eqref{eq:49} corresponds to the weak formulation of
the equation
\begin{equation}
  \label{eq:53}
  \rmL u+Vu=\lambda u\quad\text{in }\sfD,
\end{equation}
where $\rmL$ is the linear selfadjoint operator associated with $\E$.

Since $\V$ is compact in $L^2(\sfD,\mm)$ 
the standard spectral theory 
allows to prove that there exists a sequence
$\llambda(V)=(\lambda_1(V),\cdots,\lambda_k(V),\cdots)$ of eigenvalues satisfying
\begin{equation}
\lambda_{\rm min}
\leq \la_1(V)\leq \dots\leq \la_k(V)\leq \dots,\quad
\lambda_{\rm min}:=\alpha+V_{\rm min},\quad
\lim_{k\to\infty}\lambda_k(V)=+\infty,\label{eq:3}
\end{equation}
and a corresponding sequence $\uu=(u_1,u_2,\cdots,u_k,\cdots)\in \big(D(\E_V)\big)^\N$ of
eigenfunctions which provides an $L^2$-orthonormal basis for
$\overline{D(\E_V)}^{L^2}$.

Optimization problems  for eigenvalues of potentials have been recently treated in~\cite{bgrv}, and they can be used, somehow, to approximate shape optimization problems, as highlighted in~\cite[Example~5.8]{bgrv}. 

Our main structural assumption on $\V,\E,\sfD,\mm$
is that for every choice of positive constants $C,\bar\lambda\in \R^+$
the set of
eigenfunctions satisfying \eqref{eq:49}
for $\lambda\le \bar\lambda$ and $V\in \K$ with $\|V\|_{L^2(\sfD,\mm)}\le C$
is relatively compact in $L^4(\sfD,\mm)$.
This property is always satisfied if, e.g., $\V$ is compactly embedded in
$L^4(\sfD,\mm)$ or $\E$ is a Dirichlet form.

Apart from the (finite-dimensional, but still interesting) case when
$\sfD$ is
a finite set, simple relevant examples covered by our setting are
provided by a bounded Lipschitz open set $\sfD$ of $\R^d$
with the usual Lebesgue measure $\mm=\mathcal L^d$
(or a compact smooth Riemannian manifold
endowed with the Riemannian volume measure)
and
\begin{enumerate}[1.]
\item The Dirichlet (resp.~Neumann) Laplacian $\rmL u=-\Delta u$
  (the Laplace-Beltrami operator in the Riemannian case) corresponding to 
  $\dom=H^1_0(\sfD)$ (resp.~$\dom=H^1(\sfD)$) and
  $\E(u,v)=\int_\sfD\nabla u\cdot \nabla v\,\d x$.
\item The elliptic operator associated with
  the Dirichlet form
  $\E(u,v)=\int_\sfD A(x)\nabla u\cdot \nabla v\,\d x$
  in $H^1_0(\sfD)$ or $H^1(\sfD)$, 
  where $A$ satisfies the usual uniform ellipticity condition $\alpha
  |\xi|^2\le A(x)\xi\cdot \xi\le \alpha^{-1} |\xi|^2$ for some
  $\alpha>0$ and every $x\in
  \sfD,$ $\xi\in \R^d$.
  
\item The fractional Laplacian, for $s\in(0,1)$, with \[
    \E(u,v)=\int_\sfD\int_\sfD\frac{(u(x)-u(y))(v(x)-v(y))}{|x-y|^{d+2s}}
    \,\d x\,
    \d y,
  \]
  where the integral should be read in the principal value sense, and $\dom=H^s(\sfD)$ or $\dom=H^s_0(\sfD)$.
  \item
    The Dirichlet (resp.~Neumann) Bilaplacian corresponding to
    $\dom=H^2_0(\sfD)$ (or $H^2(\sfD)$) and
  $\E(u,v)=\int_\sfD\nabla^2 u: \nabla^2 v\,\d\mm$ in dimension
  $d\le 8$.
\item We can also consider the Dirichlet form induced by a nondegenerate
  Gaussian measure $\mm$ in a separable Hilbert space $\sfD$, see
  e.g.~\cite[Chap.~10]{DaPrato}.
\end{enumerate}
\paragraph{\bfseries $L^2$-gradient flows}
The aim of this paper is to study the $L^2$-gradient flow
\begin{equation}
  \label{eq:4}
  \frac\d{\d t}V(t)\in -\partial_L^-\FF(V(t))\quad t>0,\quad
  V(0)=V_0,
\end{equation}
of
potentials driven by the limiting subdifferential (known also as Mordukhovich subdifferential~\cite{krmo,mor1, mor}) of a functional
$\functional:L^2(\sfD,\mm)\to \R\cup\{+\infty\}$
arising from the sum of two competing terms $\GG$ and $\HH$:
\begin{enumerate}[(1)]
\item a convex (or a quadratic perturbation of a convex)
  confining term $\GG$ (typically nonsmooth, such as the indicator function
  of
  a closed convex set of $L^2(\sfD,\mm)$) 
  such that $\GG(V)=+\infty$ iff $V\not\in \K$,
  which in
  particular forces the potential $V$ to stay above the constant
  $V_{\rm min}$.
  $\GG$ will keep track of the class of admissible potentials (see,
  e.g., formula \eqref{eq:admi} below).
\item A term 
  \begin{equation}
    \label{eq:50}
    \HH(V):=\varphi(\lambda_1(V),\cdots,\lambda_\J(V))
  \end{equation}
  which depends on the first $\J$ eigenvalues associated with $V$ through a function
  $\varphi
  \in\rmC^1(\Lambda^\J)$
  where $\Lambda^\J $
  is the
subset of $[\lambda_{\rm min},+\infty)^\J$ spanned by all the ordered
vectors made of $\J$ real numbers.
\end{enumerate}
At least when all the first $\J+1$ eigenvalus are distinct, the gradient flow equation~\eqref{eq:4}
reads as 
\begin{equation}
V'(t)+\sum_{i=1}^\J\partial_i
\varphi(\la_1(V(t)),\dots,\la_j(V(t)))u_i^2(t)\in
-\partial^-\GG(V(t))\label{eq:5}
\end{equation}
where $(u_1(t),u_2(t),\cdots,u_\J(t))$
is an orthonormal system of eigenfunctions associated with the potential
$V(t)$
and to the eigenvalues $\lambda_i(V(t))$.
When some of the eigenvalues are multiple the function
$\HH$ loses its differentiability properties;
however, we will still be able to recover \eqref{eq:5}
at least when $\varphi$ satisfies a suitable compatibility condition
at the boundary of $\Lambda^\J$. 
Among the interesting examples that are covered 
we mention the monotonically increasing composition of symmetric
functions of the first $k$-eigenvalues, $k\le \J$,
as (here $\lambda_{\rm min}>0$)
\begin{equation}
  \label{eq:142}
  \sum_{j=1}^k \lambda_j,\quad
  \sum_{i\neq j,\, i,j\le k}\lambda_i\lambda_j,\quad
  2\lambda_1+\sqrt{\lambda_1\lambda_2\lambda_3},\quad
  (\lambda_1+\lambda_2+\lambda_3)(\lambda_1^2+\lambda_2^2).
\end{equation}
The special case
when $\varphi(\lambda_1,\cdots,\lambda_\J)=\lambda_1+\cdots+\lambda_\J$
(also with $\J=1$)
is a quite interesting example of concave functional
leading to the differential inclusion
\begin{equation}\label{gfsumj}
  V'(t)+\sum_{i=1}^\J u_i^2(t)\in -\partial_F^- \GG(V(t)).
\end{equation}
From the viewpoint of gradient flows,
the main difficulty and challenging feature arising from
increasing functions of eigenvalues is that
even the simplest map $V\mapsto \lambda_j(V)$
is not even a smooth perturbation of a convex function
with respect to the potential $V$ (when $j=1$ it is in fact a concave
function, which is not differentiable when $\lambda_1$ is a multiple eigenvalue).
Therefore many standard results of gradient flow theory do not apply.
We are thus led to follow and adapt results for gradient flows of
highly non-convex functionals proposed in~\cite{rs}.
We also have to circumvent a second
important difficulty, related to the lack of compactness of the sublevels of
$\FF$. By analyzing the structure of the limiting subdifferential of
(suitable regularizations of) $\HH$ and employing a sort of
compensated-compactness argument, we are eventually able to prove
the strong convergence of the Minimizing Movement scheme for $\FF$ and
to to show that all the limits 
satisfy \eqref{eq:4} (and \eqref{eq:5} for compatible $\varphi$).

\paragraph{\bf Plan of the paper.}

Section~\ref{sec:settingresults} is devoted to
clarify the structural assumptions we will refer in the paper.
The discussion of the main examples and of some applications
covered by the theory is carried on in Section \ref{sec:examples}.

Section \ref{sec:main} contains the precise statement
of our main results.

The crucial tools concerning the
regularity and the differentiability
properties of the functionals $\HH$ and $\FF$ are
developed in Sections \ref{sec:regularity}
and \ref{section:F} respectively.

The last Section \ref{sec:minmov} collects the main estimates
concerning the Minimizing Movement scheme
and is then devoted to the proof of its strong convergence.

The Appendix contains some basic material concerning convergence of
eigenvalues and eigenfunctions and a useful result of convex analysis.

\section{Notation and assumptions}
\label{sec:settingresults}
We briefly collect here the abstract setting 
in which we work for the whole paper and a few preliminary results.
Let $(\sfD,\mm)$ be a finite measure space with
$\Hilb:=L^2(\sfD,\mm)$ separable.
We will denote by $|\cdot|$ and $\langle\cdot,\cdot\rangle$
the norm and the scalar product of $\Hilb$.
For the sake of simplicity, in the following we will assume that $\mathrm
{dim}(\Hilb)=+\infty$; it will be easy to adapt the various statements
to the case when $\Hilb$ has finite dimension (e.g.~when $\sfD$ is a
finite set and we can identify $L^2(\sfD,\mm)$ with some $\R^d$).

\subsection*{2.A
  Closed, symmetric, and coercive bilinear forms}

We will consider a Hilbert space $\V$ satisfying
\begin{equation}
\text{$\dom\hookrightarrow \Hilb$
is densely and compactly imbedded in
$\Hilb$,}\label{eq:7}
\end{equation}
and a
\begin{equation}
  \begin{gathered}
    \text{continuous and symmetric bilinear form
      $ \E\colon \dom\times\dom\rightarrow \R$ satisfying}\\
   \E(u):= \E(u,u)\ge \alpha|u|^2,\quad
    M^{-1}\|u\|_\dom^2\le \E(u,u)+|u|^2\le
    M\|u\|_\dom^2
    \quad\text{for every }u\in \dom
  \end{gathered}
  \label{eq:8}
\end{equation}
for constants $\alpha\ge0$ and $M>0$.

The bilinear form
$  \widetilde \E(u,v)=\E(u,v)+\langle u,v\rangle
$
is a scalar product for $\dom$ inducing an equivalent norm.

\subsection*{2.B
  Admissible Shr\"odinger potentials}
  We will deal with a 
  \begin{equation}
  \begin{gathered}    
  \text{lower-semicontinuous $(-\theta)$-convex functional $\GG:\Hilb\to
    (-\infty,+\infty]$}\\
  \text{with proper domain }
  \K=D(\GG):=\big\{V\in \Hilb:\GG(V)<+\infty\big\}
\end{gathered}
  \label{eq:66}
\end{equation}
such that
\begin{equation}
  \label{eq:12}
  V\in \K\quad\Rightarrow\quad
  V\ge V_{\rm min}\ \mm\text{-a.e.~in $\sfD$},
\end{equation}%
where $\theta\ge 0,\ V_{\rm min}\in \R$ are suitable constants
that we will keep fixed throughout the paper.
Notice that the domain $\K$ of the functional $\GG$
characterizes the set of admissible potentials.

Recall that $\GG$ is $(-\theta)$-convex if
the function $\GG_\theta:V\mapsto \GG(V)+\frac
\theta2|V|^2$ is convex; for later use, we will set
\begin{equation}
  \label{eq:55}
  \Kc c :=\Big\{V\in \K: |V|\le c,\
  \GG_\theta(V)
  \le
  c\Big\}\quad c\ge 0,
\end{equation}
so that $\Kc c $, $c\ge 0$, is an increasing family of closed and bounded convex
subsets of $\Hilb$ 
whose union is $\K$. Since $\K$ is not empty, it contains at least an
element $V_o$, so that setting $c_o:=\max(|V_o|,\GG_\theta(V_o))$ the set
$\Kc c$ is not empty for every $c\ge c_o$.

\subsection*{2.C 
  Eigenvalues}
For every $V\in \K$
we consider the symmetric bilinear form
\begin{equation}
  \label{eq:9}
  \E_V(u,v):=\E(u,v)+\int_\sfD V\,uv\,\d\mm,\quad
  D(\E_V):=\Big\{u\in \dom:\int_\sfD V_+u^2\,\d\mm<+\infty\Big\}.
\end{equation}
Denoting by $\Hilb_V$ the closure of $D(\E_V)$ in $\Hilb$,
it is clear that $\E_V$ is a closed and symmetric bilinear form, whose
domain $D(\E_V)$ is compactly imbedded in $\Hilb_V$.
It is also clear that 
\begin{equation}
  \label{eq:67}
  \V_4:=\V\cap L^4(\sfD,\mm)\subset D(\E_V)\quad
\text{for every }V\in \K.
\end{equation}
We say that $\lambda\in \R$ is an eigenvalue for (the linear operator induced by)
$\E_V$
if there exists $u\in D(\E_V)\setminus \{0\}$ such that
\begin{equation}
  \label{eq:10}
  \E(u,v)+\int_\sfD V\,uv\,\d\mm=\lambda\int_\sfD uv\,\d\mm\quad
  \text{for every }v\in D(\E_V).
\end{equation}
Any nontrivial solution $u$ to \eqref{eq:10} is called a
$(V,\lambda)$-eigenfunction
($u$ is also normalized if $|u|=1$).

The standard spectral theory applied to the bilinear form $\E_V$ 
allows us to prove that there exists a sequence
$\llambda(V)=(\lambda_1(V),\cdots,\lambda_k(V),\cdots)$ of eigenvalues
satisfying
\begin{equation}
\lambda_{\rm min}
\leq \la_1(V)\leq \dots\leq \la_k(V)\leq \dots,\quad
\lambda_{\rm min}:=\alpha+V_{\rm min},\quad
\lim_{k\to\infty}\lambda_k(V)=+\infty.\label{eq:3bis}
\end{equation}
Notice that $\lambda_{\rm min}$ is defined in terms of $\alpha$ and
$V_{\rm min}$ and it will remain fixed throughout the paper.
Moreover the sequence of the eigenvalues
can be characterized by means of a min-max principle,
\begin{equation}
  \label{eq:21}
  \begin{aligned}
    \la_j(V)&=\min_{E_j \subset D(\E_V)}\max_{u\in E_j\setminus
      \{0\}}\frac{\E(u)+\int_\sfD{V
        u^2\,\d\mm}}{\int_\sfD{u^2\,\d\mm}}
    \\&=\min_{E_j \subset D(\E_V)}\max_{u\in
      E_j\setminus \{0\}}\left\{
      \E(u)
      +\int_\sfD{V
        u^2}\,\d\mm\;:\;\int_\sfD{u^2}\,\d\mm=1\right\},
  \end{aligned}
\end{equation}
where the minimum is taken over the subspaces $E_j\subset
D(\E_V)$ of dimension $j$.
For a given $\J\in \N$
$\llambda^\J=\llambda^\J(V)\in
\R^\J$ will denote the vector of the first $\J$ eigenvalues;
we will denote by $\uU^\J(V)$ the collection of all the orthonormal systems of eigenfunctions associated with $\llambda^\J(V)$, namely
\begin{equation}
  \label{eq:11}
  \begin{aligned}
    \uU^\J(V):=\Big\{&\uu=(u_1,u_2,\cdots,u_\J)\in D(\E_V)^\J:
    u_i\text{ is a normalized $(V,\lambda_i)$-eigenfunction,}\\
    &\int_\sfD u_iu_j\,\d\mm=0\text{ for every }1\le i,j\le \J,\ i\neq
    j\Big\}.
  \end{aligned}
\end{equation}
We note that the eigenvalues satisfy the following monotonicity property with respect to the potential:
\begin{equation}\label{eq:moneig}
  \text{if }V_1\leq V_2\ \text{$\mm$-a.e. in $\sf D$}
  \quad\text{then}\quad  \la_k(V_1)\leq\la_k(V_2)\quad\text{for all $k\in\N$.}
\end{equation}

\subsection*{2.D
  $L^4$-summability of eigenfunctions}
We will assume that
  every eigenfunction $u$ solving \eqref{eq:10} for some $V\in \K$
  belongs to $L^4(\sfD,\mm)$ and for every constant $c\ge c_o$ and
  $\bar \lambda>\lambda_{\rm min}$ the set
  \begin{equation}
    \label{ass:main}
    \begin{gathered}
      \mathrm U[c,\bar \lambda]:=\Big\{u\text{ is a normalized
        $(V,\lambda)$-eigenfunction with } V\in \Kc c ,\ \lambda\le
      \bar\lambda\Big\}\\
      \text{  is relatively compact in $L^4(\sfD,\mm)$.}
    \end{gathered}
  \end{equation}%
\begin{remark}
  \label{rem:equivalent}
Clearly $\K[c]$ is weakly compact in $\Hilb$. Further, we will see that
  from every sequence $u_n$ of $(V_n,\lambda_n)$-eigenfunctions, $n\in \N$, 
  with $\lambda_n\to \lambda$ and $V_n\weakto V$ in $\Hilb$
  it is possible to extract a subsequence $k\mapsto u_{n(k)}$ 
  strongly converging to a $(\lambda,V)$-eigenfunction $u$ in $\V$. Then,~\eqref{ass:main} is in fact equivalent to the compactness of $\mathrm U[c,\bar
  \lambda]$ in $L^4(\sfD,\mm)$ and can also be formulated as a continuity property:
  \begin{equation}
    \label{eq:143}
    \begin{gathered}
      \text{for every sequence $(u_n)_{n\in \N}$ of normalized
  $(V_n,\lambda_n)$-eigenfunctions:}\\  
      V_n\weakto V\text{ in }\Kc c,\ 
      \lambda_n\to\lambda,\
      u_n\to u\text{ strongly in }\V\quad\Rightarrow\quad
      u_n\to u\ \text{strongly in }L^4(\sfD,\mm).
    \end{gathered}
  \end{equation}
\end{remark}

Assumption \eqref{ass:main}
and \eqref{eq:67}
guarantee that every $(V,\lambda)$-eigenfunction $u$
belongs to $\V_4\subset D(\E_W)$ for every $W\in \K$ and
\begin{gather}
  \label{eq:58}
    \V_4=\V\cap L^4(\sfD,\mm)\text{ is dense in $D(\E_V)$
      for every $V\in \K$},\\
    \label{eq:68}
    \Hilb_4:=\overline{\V_4}^{\Hilb}=
    \overline {D(\E_V)}^{\Hilb}\quad\text{for every
    }V\in \K,\\
    \label{eq:69}
    \text{$\V_4$ is dense in $\V$ and $\Hilb_4=\Hilb$}\quad
    \text{if\quad
      $\K\cap L^\infty(\sfD,\mm)\neq \emptyset.$}
\end{gather}
In fact, if $V\in \K$ and $(u_k)_{k\in \N}$ 
is an orthonormal system of
the eigenfunctions of $\E_V$,
the space $E:=\Span\{u_k:k\in \N\}$
is contained in $\V_4$ and it
is dense in $D(\E_V)$.
In particular 
$\overline E^{\Hilb}=\overline{\V_4}^{\Hilb}=\overline{D(\E_V)}^{\Hilb}$.
If $V\in \K\cap L^\infty(\sfD,\mm)$, then
$D(\E_V)=\V$ which is dense in $\Hilb$.

Moreover, Assumption \eqref{ass:main} yields in particular that
for every $c\ge c_o,\ \bar\lambda>\lambda_{\rm min}$ there exists
a constant $C$ such that
\begin{equation}
  \label{eq:56}
  u\in \mathrm U[c,\bar\lambda]\quad\Rightarrow\quad
  \|u\|_{L^4(\sfD,\mm)}\le C.
\end{equation}
Conversely, if for every $c\ge c_o,\ \bar\lambda>\lambda_{\rm min}$
there exists $p>4$ and $C>0$ such that
\begin{equation}
  \label{eq:57}
  u\in \mathrm U[c,\bar\lambda]\quad\Rightarrow\quad
  \|u\|_{L^p(\sfD,\mm)}\le C,
\end{equation}
then Assumption \eqref{ass:main} is satisfied, since $\mathrm
U[c,\bar\lambda] $ is clearly bounded in $\V$ (thus relatively compact
in $L^2(\sfD,\mm)$) and bounded in $L^p(\sfD,\mm)$ for $p>4$,
and therefore relatively compact in $L^4(\sfD,\mm)$. 
Here are a few simple examples where Assumption \eqref{ass:main} is satisfied:
\begin{enumerate}[(a)]
\item $\sfD$ is a finite set.
\item $\E$ is a Dirichlet form.
\item $\dom$ is continuously embedded in $L^4(\sfD,\mm)$.
\item For every $c\ge c_o$ $\K[c]$ is bounded in $L^\infty(\sfD,\mm)$ (so that $D(\E_V)=\dom$ and
  $\mathrm L_V=\mathrm L+V$ for every $V\in \K$)
  and the resolvent operator
  $(\mathrm I+\mathrm L)^{-1}$ is bounded from $L^2(\sfD,\mm)$
  to $L^p(\sfD,\mm)$ for some $p>4$.
\end{enumerate}
We will discuss these cases in the next Section \ref{sec:examples}

\subsection*{2.E
  Functionals depending on the first $\J$ eigenvalues}
We denote by $\Lambda^\J$ the
subset of $\R^\J$ spanned by all the ordered vectors made of $\J$ real numbers, namely, 
\begin{equation}
  \label{eq:6}
  \Lambda^\J:=\Big\{(\lambda_1,\cdots,\lambda_\J)\in
  \R^\J:
  \lambda_{\rm min}\le
  \lambda_1\le \lambda_2\le \cdots\,\le \lambda_\J\Big\}.
\end{equation}
This can be seen as the set of all the possible first $\J$ eigenvalues.
Notice that $\Lambda^\J$ is a closed convex subset of $\R^\J$ whose
interior in $[\lambda_{\rm min},+\infty)^\J$ is
\begin{equation}
  \label{eq:131}
  \Int(\Lambda^\J)=
  \Big\{(\lambda_1,\cdots,\lambda_\J)\in
  \R^\J:
  \lambda_{\rm min}\le\lambda_1< \lambda_2< \cdots\,< \lambda_\J\Big\}.
\end{equation}
We consider a function
$\varphi:\Lambda^J\to \R$
of class $\rmC^1$ satisfying
\begin{equation}
  \label{eq:65}
  \quad
  \varphi(\llambda)\ge -A(1+\lambda_J\lor 0)
  \quad\text{for every }\llambda\in \Lambda^J,
\end{equation}
for some constant $A\ge0$.
We can then define the functionals
\begin{equation}
  \label{eq:2}
  \HH:\K\to \R,\quad
  \HH(V):=\varphi(\llambda^\J(V))\quad\text{for every }V\in \K
\end{equation}
and
\begin{equation}
  \label{eq:13}
  \FF:\Hilb\to (-\infty,+\infty],\quad
  \FF(V):=\HH(V)+\GG(V)\text{ if }V\in \K,\quad
  \FF(V):=+\infty\text{ if }V\not\in \K.
\end{equation}
\subsection*{2.F
  Additional structural compatibility}
Our main existence result will only rely on the assumptions
2.A--2E we have discussed above. 
To gain a more refined characterization of
the gradient flows, we will sometimes assume
a further structural property on $\varphi$ in addition to $\rmC^1$ regularity.
More precisely, we will suppose that
for every 
$\llambda=(\lambda_1,\cdots,\lambda_\J) \in \Lambda^\J$
  \begin{equation}
    \label{eq:29}
    \begin{gathered}
      \partial_{J}\varphi(\llambda)\ge0,\\
      k\in
      \{2,\cdots,\J\},\ \lambda_{k-1}=\lambda_k\quad \Rightarrow\quad
      \partial_{k-1}\varphi(\llambda)\ge \partial_k\varphi(\llambda).
    \end{gathered}
\end{equation}
Notice that if $\varphi$ is considered as
a function
depending on the first $K$ eigenvalues
with $K>\J$ (and therefore is trivially extended to $\Lambda^{K}$),
then \eqref{eq:29} also holds on $\Lambda^{K}$.
Conversely, if $\varphi$ just depends on the first $I$ eigenvalues,
$I<J$, then $\partial_k\varphi(\llambda)=0$ for $k>I$ and
\eqref{eq:29} yields $\partial_I\varphi(\llambda)\ge0$.

If  \eqref{eq:29} holds, then 
for every $k\in
\{1,\cdots,\J\}$ and every $(\lambda_1,
\cdots \lambda_{k-1})\in \Lambda^{k-1}$ the map 
\begin{equation}
  \label{eq:91b}
  z\mapsto \varphi(\lambda_1,\cdots,\lambda_{k-1},z,\cdots,z)
  \quad\text{is increasing in }[\lambda_{k-1},+\infty),
\end{equation}
so that
\begin{align*}
  \varphi(\lambda_{\rm min}, \lambda_{\rm min},\cdots, \lambda_{\rm
  min})
  &\le
    \varphi(\lambda_1, \lambda_1,\cdots, \lambda_1)
    \le
    \varphi(\lambda_1, \lambda_2,\cdots, \lambda_2)
    \le
    \varphi(\lambda_1, \lambda_2,\lambda_3,\cdots, \lambda_3)
    \le \cdots
    \\&\le \varphi(\lambda_1, \lambda_2,\cdots, \lambda_J);
\end{align*}
therefore $\varphi$ is bounded from below
and \eqref{eq:65} holds as well with
$A:=\big(\varphi(\lambda_{\rm min}, \lambda_{\rm min},\cdots, \lambda_{\rm
  min})\big)_-$.

Clearly \eqref{eq:29} characterizes a convex cone in $\rmC^1(\Lambda^\J)$.
It is not difficult to check that if 
$\phi\in \rmC^1([\lambda_{\rm min},+\infty)^K)$, $K\le \J$, is a
symmetric function satisfying
\begin{equation}
  \label{eq:30}
  \begin{gathered}
    \phi(\sigma(\llambda))=\phi(\llambda)\quad\text{for every
    }\llambda\in [\lambda_{\rm min},+\infty)^K,\quad \sigma\in
    \operatorname{Sym}(K),\quad
    \big(\sigma(\llambda)\big)_k=\llambda_{\sigma(k)},\\
    \partial_K\phi(\llambda)\ge0,
    \end{gathered}
\end{equation}
then
$\varphi(\lambda_1,\cdots,\lambda_\J):=\phi(\lambda_1,\cdots,\lambda_K)$
satisfies \eqref{eq:29}.

So, if e.g.~$\lambda_{\rm min}>0$, monotone compositions of
symmetric functions such as
\begin{displaymath}
  2\lambda_1+\sqrt{\lambda_1\lambda_2}+
  \ln(\lambda_1\lambda_2\lambda_3),\quad
  (\lambda_1+\lambda_2+\lambda_3)(\lambda_1^2+\lambda_2^2)
\end{displaymath}
satisfy \eqref{eq:29}.
Other examples are provided by the functions
\begin{displaymath}
  \llambda\mapsto h(\lambda_j-\lambda_{j-1}),\quad
  1\le j\le \J,\quad h\in \rmC^1([0,+\infty)),\quad
  0=h'(0)\le h'(r)\text{ for every }r\ge0.
\end{displaymath}
\begin{remark}
  From now on we will always operate in the setting described in
  2.A--2.E;
  in particular, the constants $\alpha, M, \theta,V_{\rm min},
  \lambda_{\rm min},c_o,\J$ will be considered as fixed throughout the paper.  
\end{remark}
\section{Examples and applications}
\label{sec:examples}
Let us briefly show a few examples where
the assumptions of Section 2.A--2.D apply, considering in particular
the cases (a,\ldots,d) of 2.D.
\subsection{The finite dimensional case}
Let $\sf D$ be a finite set which we can identify with
the set of indices $\{1,\cdots,d\}$
so that $\Hilb=L^2(\sf D,\mm)$ can be identified to $\R^d$ for some $d\geq 1$.
In this case the bilinear form $\E$ can be identified with a $d\times d$ matrix $L=(a_{ij})$, symmetric and nonnegative definite, so that \[
  \E(u,v)=\sum_{i,j=1}^d a_{ij}u_iv_j=\langle L u,v\rangle\qquad
  \text{for $u,v\in \R^d$}.
\]
We can take as 
$\GG$ any convex and lower semicontinuous function in $\R^d$
whose proper domain is a closed convex set
$\K\subset [v_{\rm min},+\infty)^d$, for some constant $v_{\rm min}\in \R$.

For every $V\in \K$ we have the symmetric bilinear form \[
  \E_V(u,v):=\E(u,v)+\sum_{i=1}^dV_iu_iv_i=
  \langle Lu+Vu,v\rangle,\quad
  (Vu)_i:=V_i u_i,
\]
and $\lambda\in \R$ is a $(\lambda,V)$-eigenvalue if
there exists $u\in \R^d \setminus \{0\}$ such that \[
  Lu+Vu=\lambda u.
\]
Since we are in a finite dimensional setting,
Assumption~\eqref{ass:main} is trivially satisfied.

\subsection{The case of a Dirichlet form}\label{sec:dirform}
Let us now consider the case when $\E$ is a Dirichlet form, thus 
satisfying the Markov condition
(see~\cite{fukushima})
\begin{equation}\label{eq:markovian}
\text{for all $u\in \V$, then $v:=\min\{\max\{u,0\},1\}\in \V$ and $\E(v)\leq \E(u)$.}
\end{equation}
Since, for every $V\in \K$, $\V\cap L^\infty(\sfD,\mm)\subset D(\E_V)$
and $\V\cap L^\infty(\sfD,\mm)$ is also dense in $\V$ and thus in
$\Hilb$, we deduce that $D(\E_V)$ is dense in $\Hilb$.
If $\beta:=1+(\lambda_{\rm min})_-$ the quadratic form
$\E_V+\beta|\cdot|^2$ is a Dirichlet form associated with the selfadjoint
operator
by $\rmL_V+\beta$ whose inverse 
$\mathrm R_V^\beta:=(\rmL_V+\beta)^{-1}:\Hilb \to\Hilb$
is a sub-Markovian compact selfadjoint operator
(and in particular a
contraction) in $\Hilb$;
it is well known that
$u$ is a $(\lambda,V)$ eigenfunction if and only if
$u$ is a $(\lambda+\beta)^{-1}$-eigenfunction of the operator
$\rmR_V^\beta$ (see also Appendix \ref{sec:appmosco}).
The restriction of $\rmR_V^\beta$ to $L^p(\sfD,\mm)$ is also a contraction.
By \cite[Theorems 1.6.1-2-3]{davies}
the spectrum and the eigenfunctions of $\rmR_V^\beta $ are
independent of $p\in [2,+\infty)$ and in particular all the eigenfunctions associated with
$V$ belong to $L^p(\sfD,\mm)$ for every $p\in [2,+\infty)$.

Let now $u_n$ be a sequence of
normalized $(V_n,\lambda_n)$-eigenfunctions with
$\lambda_n\le \bar\lambda$ and $V_n\in \K[c]$.
Up to extracting a suitable subsequence, it is not restrictive to
assume that $\lambda_n\to \lambda$, $V_n\weakto V$ in $\Hilb$,
$u_n\to u$ strongly in $\V$ for some
$\lambda\le \bar \lambda$, $V\in \K[c]$, $u\in D(\E_V)$.

By Lemma \ref{le:A1} in the Appendix,
$\rmR_{V_n}^\beta$ converge
uniformly to $\rmR_V^\beta $ in $\mathcal L(\Hilb)$: this implies that
$u$ is a normalized $(V,\lambda)$-eigenfunction.
We want to show that $\|u_n-u\|_{L^4}\to 0$; we fix $p>4$ and we show
that the $L^p$-norm of $u_n$ is bounded.
We argue by contradiction, assuming that $\|u_n\|_{L^p}\to+\infty$
along a (not relabeled) subsequence.

Since $\rmR_V$ and $\rmR_{V_n}$ are contractions in $L^q(\sfD,\mm)$ for any $q>p>4$,
by Riesz-Thorin interpolation we also have that
$\|\rmR_V^\beta -\rmR_{V_n}^\beta \|_{\mathcal L(L^p)}\to0$ as $n\to+\infty$.
Setting $\tilde u_n:=\|u_n\|_{L^p}^{-1}u_n$
by \cite[Thm. 7.4, p. 690]{Babuska-Osborn}
we find a $(V,\lambda)$-eigenfunction $v_n$ such that
$\|\tilde u_n-v_n\|_{L^p}\to0$.
We deduce that $\|v_n\|_{L^p}$ is bounded; since $v_n$ belongs to a
finite dimensional space, it admits a subsequence $v_{n(k)}$ strongly
convergent to some limit $v$ in $L^p(\sfD,\mm)$.
Therefore $\tilde u_n\to v$ strongly in $L^p(\sfD,\mm)$ with $\|v\|_{L^p}=1$; however
$\tilde u_n\to0$ in $L^2(\sfD,\mm)$, a contradiction.
\begin{remark}
  All the examples 1,2,3,5 considered in the Introduction fit in the
  framework of Dirichlet forms, with domain $\V$ which is compact in $L^2(\sfD,\mm)$.
\end{remark}

\subsection{The case when $\V\subset L^4(\sfD,\mm)$ or the resolvent has
a regularizing effect}
This case follows immediately from the equivalent characterization
\eqref{eq:143}
of \eqref{ass:main}. Notice that the example 4 in the Introduction
corresponds to this situation, thanks to the Sobolev imbedding of
$H^2(\sfD)$ in $L^4(\sfD)$ when the dimension $d\le 8$.

The last case (d) considered in Section 2.D can be easily discussed by
observing that if $V\in L^\infty(\sfD,\mm)$ then a normalized
$(V,\lambda)$-eigenfunction $u$ satisfies
the equation
\begin{displaymath}
  \rmL u+u=f\quad
  \text{with }
  f:=\lambda u-Vu;
\end{displaymath}
since $\|f\|_{L^2}\le |\lambda|+\|V\|_{L^\infty}$, we immediately
recover a uniform estimate of $u$ in $L^p(\sfD,\mm)$ as in \eqref{eq:57}.

\medskip\noindent
We conclude this section by briefly discussing two possible
applications of gradient flows of spectral functionals.

\subsection{Optimal design problems arising from population dynamics and reaction diffusion equations}
Let
$\sfD\subset\R^d$ be an open
and bounded set, let 
$V^-<v_0<V^+\in\R$ be real constants with $V^-<0<V^+$, 
and let
$$\K:=\Big\{V\in L^\infty(\sfD;\mm):V^-\le V\le V^+,\
\mean\sfD V(x)\,\d x\geq v_0
\Big\}.$$
For $V\in \K$, the classical reaction-diffusion model in an heterogeneous environment proposed by Fisher and Kolmogorov can be generalized as:
\begin{equation}\label{FK}
\left\{
\begin{split}
u_t&=\Delta u-uV(x)-u^2,\qquad &&\mbox{in }\sfD\times\R^+,\\
u&=0,\quad\text{(or $\partial_\nu u=0$)},\qquad &&\mbox{on }\partial \sfD\times \R^+,\\
u(x,0)&\geq 0, \qquad u(x,0)\not \equiv 0,\qquad &&\mbox{in }\overline \sfD,
\end{split}
\right.
\end{equation}
where $u(x,t)$ represents the population density at time $t$ and position $x$, and $(-V(x))$ is the intrinsic grow rate of the species at the spatial point $x$. 
The condition proved in~\cite{bhr} for the survival of the species for large times (as $t\rightarrow \infty$) is that the first eigenvalue $\la_1(V)$ for the associated linearized problem (we stress that here we have the opposite sign in front of the potential, with respect to~\cite{bhr}) which is defined as\[
\left\{
\begin{split}
-\Delta u+Vu&=\lambda_1(V)u,\qquad &&\mbox{in }\sfD,\\
u&=0, \quad\text{(or $\partial_\nu u=0$)},\qquad &&\mbox{on }\partial \sfD,
\end{split}
\right.
\]
should be negative, so it is natural to try to minimize it under the
constraint $V\in \K$.
This problem has been widely studied (see for example~\cite{cc,llnp}
and the references therein): it is known that an optimal potential
$V^*$ is of bang-bang type, i.e. $V^*=V^+$ on $\sfD^+\subset \sfD$ and
$V^*=V^-$ on $\sfD^-=\sfD\setminus \sfD^+$. On the other hand, there
are still many open problems concerning the shape of the partition
$\sfD^\pm$.
The $L^2$-gradient flow of the functional
\begin{displaymath}
  \FF(V):=
  \begin{cases}
    \lambda_1(V)&\text{if }V\in \K,\\
    +\infty&\text{otherwise},
  \end{cases}
\end{displaymath}
can provide some useful new insights.
\subsection{{Optimization of eigenvalues of potentials}} 
In the paper~\cite{bgrv} some optimization problems for eigenvalues of
potentials in the case of the Dirichlet Laplacian, i.e. when
$\sfD\subset\R^d$ is an open and bounded set, $\dom=H^1_0(\sfD)$ and
$\E(u,v)=\int_\sfD\nabla u\cdot\nabla v\,\d x$,  were considered. 
The authors studied the minimization problem 
\begin{equation}\label{eq:potopt}
\min\Big\{\varphi(\la_1(V),\dots,\la_\J(V)) : V\in \K\Big\},
\end{equation}
for all $\varphi\colon \R^\J\to \R$ regular and increasing in each
variables,
with the class of admissible potentials $\K$
defined as follows:
\begin{equation}\label{eq:admi}
  \K=\left\{V\in L^2(\sfD,\mm):V\ge 0,\  \mean \sfD{\Psi(V)\leq c}\right\},
\end{equation}
where $\Psi: [0,+\infty]\rightarrow[0,+\infty]$ 
denotes a strictly decreasing convex function 
and
$c\in \R$ is such that
\[
 \lim_{r\to+\infty}\Psi(r) < c < \Psi(0).
\]
It is clear that $\K$ is convex and closed in $L^2(\sf D,\mm)$.
Some remarks about the choice of the class of potentials are in order.
First of all, we note that examples of function satisfying the
hypotheses above are $\Psi(s)=s^{-\beta}$ or $\Psi(s)=e^{-\beta s}$,
for some $\beta>0$. It is immediate to check that $\K$ is not
empty
and that $0\not \in \K$, so that
the trivial potential $V=0$ is not allowed. 

By choosing $\GG$ as the indicator function of $\K$ (which is clearly convex and lower semicontinuous)
and $\HH(V)=\varphi (\la_1(V),\dots, \lambda_j(V))$,
we provide a gradient flow evolution for the minimization problems studied in~\cite[Section~4]{bgrv}. We note that in our $L^2$ setting, the existence of minimizers for the problem \[
\min{\Big\{\varphi(\lambda_1(V),\dots, \lambda_\J(V)) : V\in \K\Big\}},
\]
follows easily since the functional is weakly lower semicontinuous in $L^2(\sf D,\mm)$.

When $\Psi(s)=e^{-\beta s}$ the interest for problem~\eqref{eq:potopt}
also lies in the fact that
it can be used as an approximation of a shape optimization problem (see~\cite[Example~5.8]{bgrv}), namely \[
\min\Big\{\varphi(\la_1(\Omega),\dots,\la_\J(\Omega)) : \Omega\subset \sfD,\; \mm(\Omega)=c\leq \mm(\sf D)\Big\}.
\]

\section{Main results}
\label{sec:main}
In order to make precise the notion of gradient flows we are going to
study, let us first recall the main definitions of
subdifferentials which are involved. We refer to~\cite[Chap. 8]{rowe} for more details.
\begin{definition}[Fr\'echet and limiting subdifferentials]\label{subdiff}
Let $\tFF:\Hilb\to \R\cup\{+\infty\}$ 
and let $v\in D(\tFF)\subset \Hilb,\ \xi\in \Hilb.$
We say that 
$\xi$ belongs to the \emph{Fr\'echet subdifferential} $\partial_F^-\tFF(v)$
if
\[
  \begin{aligned}
    \liminf_{w\rightarrow v}\frac{\tFF(w)-\tFF(v)-\langle \xi,
      w-v\rangle }{|w-v|}&\geq 0;
\end{aligned}
\]
equivalently, by using the viscosity characterization
\cite[Remark 1.4]{Borwein-Zhu99},
there exist $\varrho>0$ and a function
\begin{equation}
  \label{eq:barrier}
  \omega:\Hilb\to[0,+\infty)\quad\text{of class $\rmC^1$, convex, and
    satisfying}
  \quad
  \omega(0)=0,
\end{equation}
such that 
\begin{equation}
  \tFF(w)-\tFF(v)-\langle \xi,w-v\rangle \geq -\omega(w-v)\quad 
\text{for every }w\in \rmB(v,\varrho)=\{w\in \Hilb : |w-v|<\rho\}.\label{eq:136}
\end{equation}
$\xi$ belongs
to the \emph{limiting subdifferential} (known also as \emph{Mordukhovich subdifferential}~\cite{krmo,mor1,mor}) $\partial_L^-\tFF(v)$ if there exist sequences $v_n,\xi_n\in \Hilb$ such that 
\begin{equation}\label{eq:limitsubdiff}
\xi_n\in \partial_F^-\tFF(v_n),\quad v_n\rightarrow v\;\text{strongly
  in $\Hilb$},\quad \xi_n\rightharpoonup \xi\;\text{weakly in
  $\Hilb$},\quad \tFF(v_n)\to \tFF(v).
\end{equation}
We denote by $\partial_{L}^{\circ}\tFF(v)$ the element of minimal norm in $\partial_{L}^-\tFF(v)$.
\end{definition}
\begin{remark}[On the definition of Fr\'echet and limiting subdifferential]
  \label{rem:tedious}
  If we restrict the functions $\omega$ to the class
  $\omega(\delta):=
  q|\delta|^2$ then \eqref{eq:136} corresponds to the definition of
  proximal subdifferential.
  Notice that here we adopted a definition of
  limiting subdifferential which is stronger than
  the one considered in \cite{rs} (and denoted by $\partial_\ell^-\tFF$), since
  in \eqref{eq:limitsubdiff} 
  we require the convergence of the functionals 
  $\tFF(v_n)\to \tFF(v)$ instead of their boundedness.
  This choice is justified by the better regularity properties of
  the functionals which we are considering,
  but in the case of $\FF$ the two definition will lead to the same object.
\end{remark}
It is well known that when $\tFF$ is $(-\eta)$-convex and
lower semicontinuous, then $\partial_F^-\tFF$ and
$\partial_L^-\tFF$ coincide
\cite{Clarke} and can also be characterized by
  \begin{equation}
  \label{eq:73}
  \xi\in \partial_F^-\tFF(v)
  \quad\Leftrightarrow\quad
  \tFF(w)\ge \tFF(v)+\langle \xi,w-v\rangle-
  \frac\eta 2|w-v|^2\quad
  \text{for every }w\in \Hilb.
\end{equation}
In particular, when $\eta=0$ and $\tFF$ is convex we recover 
the usual subdifferential
of convex analysis which we will simply denote by $\partial^-\tFF$.

For a given time interval $[0,T]$, $T>0$, 
the gradient flow of a convex functional then reads as the solution
$v:[0,T]\to \Hilb$ of the differential inclusion
\begin{equation}
  \label{eq:127}
  v'(t)\in -\partial_F^-\tFF(v(t))\quad\text{for a.e.~$t\in (0,T)$},
\end{equation}
and for every given initial condition $v_0\in D(\tFF)$
there exists a unique solution $v\in H^1(0,T;\Hilb)$ satisfying
\eqref{eq:127} and $v(0)=v_0$ \cite{Brezis73}.

In our case, the interesting functionals $\FF$ are typically neither convex nor
$(-\eta)$-convex for any choice of $\eta>0$, since even simple examples
such as $\HH(V)=\sum_{j=1}^J\lambda_j(V)$ are nonsmooth concave
functionals. In this case the graph of the proximal (and also of the Fr\'echet) subdifferential is
not closed
and it is then natural to study the corresponding equation
of \eqref{eq:127} in terms of the limiting subdifferential (see
e.g.~the discussion in \cite{rs}).
A further difficulty arises by the fact that we did not assume any
compactness on the sublevels of $\FF$.

In order to circumvent these difficulties, we adopt the variational
approach of the Minimizing Movement method
\cite{rs,ags}, trying to obtain the gradient flow as a limit
of a discrete approximation.

We introduce a uniform partition of the interval $[0,T]$:\[
  0=t_0<t_1<\dots<t_{N-1}<T\leq t_N,\quad
  t_n:=n\tau,\ n\in \{0,\cdots,N\},\ N=N(\tau):=\lceil T/\tau\rceil
\]
corresponding to a step size $\tau>0$, the 
perturbed functionals
\begin{equation}
  \label{eq:104}
  \Phi(\tau,V;W):=\FF(W)+\frac 1{2\tau}|W-V|^2,\quad V,W\in \Hilb,
\end{equation}
and we consider the
discrete solutions $\{V^n_\tau\}_{n\in \N}$ in $\Hilb$ of the
variational iterative scheme
starting from a given initial datum $V_0\in D(\FF)$:
\begin{equation}\label{impleuler}
  V^{n}_\tau\in \argmin_{V\in \Hilb}\Phi(\tau, V^{n-1}_\tau;V),
  \quad
  n=1,\cdots,N(\tau),\quad
  V^0_\tau:=V_0 \in D(\FF).
\end{equation}
We will show (see Lemma \ref{le:Flsc}) that
a discrete solution always exists for every initial datum
  $V_0\in D(\FF)$.

We then call $\bar V_\tau$ the piecewise constant interpolant
and 
by $V_\tau$ the piecewise linear interpolant
of the discrete values $\{V^n_\tau\}$:
\begin{equation}\label{eq:interpolantlpc}
  \bar V_\tau(t)=V^n_\tau,\qquad V_\tau
  (t):=\frac{t_n-t}{\tau}V^{n-1}_\tau+\frac{t-t_{n-1}}{\tau}V^n_\tau\qquad
  \text{if }t\in(t_{n-1},t_n].
\end{equation}

\begin{definition}[Generalized Minimizing Movements]
  \label{def:GMM}
We say that a curve $V\colon [0,T]\rightarrow \Hilb$ is a
(strong) Generalized Minimizing Movement
for $\Phi$ starting from $V_0\in \Hilb$ in $[0,T]$ if there exists a
decreasing vanishing sequence of step sizes $(\tau(k))_{k\in \N}$,
$\tau(k)\downarrow0$ as $k\to\infty$, 
and a corresponding sequence of discrete
solutions $\bar V_{\tau(k)}$ such that
\begin{equation}
  \label{eq:129}
  V_{{\tau}(k)}(t)\to V(t)\quad\text{strongly in $\Hilb$ for every $t\in [0,T]$.}
\end{equation}
We denote by
$\mathrm{GMM}(\Phi, V_0,T)$
the collection of all the (strong) Generalized Minimizing Movements for $\Phi$
starting from $V_0$
in the interval $[0,T]$.
\end{definition}

Our first result reads as follows.
\begin{theorem}
  \label{thm:main1}
  In the setting of Section \ref{sec:settingresults} under
  the assumptions stated in 2.A--2.E,
  for every choice of $V_0\in \K$
  the set $\mathrm{GMM}(\Phi,V_0,T)$ is not empty.
  Every $V\in \mathrm{GMM}(\Phi,V_0,T) $ belongs to $H^1(0,T;\Hilb)$, it satisfies for almost every $t\in (0,T)$
\begin{equation}\label{eq:127ter}
\text{$V'(t)$ is the projection of the origin on the affine hull $\mathrm{aff}\Big(-\partial_L^-\FF(V(t))\Big)$},
\end{equation}
it solves the Cauchy problem, 
  \begin{equation}
  \label{eq:127bis}
  V'(t)= -\partial_L^\circ\FF(V(t))\quad\text{for a.e.~$t\in
    (0,T)$},\quad
  V(0)=V_0,
\end{equation}
and satisfies the Energy-Dissipation Identity
\begin{equation}
  \label{eq:128}
  \FF(V(t))=\FF(V_0)-\int_0^t |V'(s)|^2\,\d s\quad\text{for every
  }t\in [0,T]
\end{equation}
Finally, if $k\mapsto\tau(k)$ is a vanishing sequence as in
\eqref{eq:129} we also have
\begin{equation}
  \label{eq:130}
  V_{\tau(k)}'\to V'\quad\text{strongly in }L^2(0,T;\Hilb),\quad
  \FF(\bar V_{\tau(k)}(t))\to \FF(V(t))\quad\text{for every }t\in [0,T].
\end{equation}
\end{theorem}
\begin{remark}[Affine projection and minimal selection]\label{minimalselection}
We recall that the affine hull of a set $A\subset\Hilb$ is defined as \[
{\mathrm{aff}}(A)=\Big\{\sum_it_ia_i : t_i\in\mathbb{R},\;\sum_it_i=1,\;a_i\in A\Big\}.
\] 
Notice that we have retrieved the minimal section
principle~\eqref{eq:127bis} (even in the stronger formulation~\eqref{eq:127ter}) in this non convex case:
though in general $\partial_L^-\FF(V(t))$ is not convex, $V'(t)$ is its unique element of minimal norm for a.e. $t\in(0,T)$.
\end{remark}
Under the sole $\rmC^1$ assumption on $\varphi$ of Section 2.E,
the precise characterization of
$\partial_L^\circ\FF(V(t))$ is not immediate.
A first piece of information is provided by the following proposition.
\begin{proposition}
  \label{prop:better}
  Let $V$ be a solution to \eqref{eq:128} and let us denote by $O\subset [0,T]$ the open set
  \begin{equation}
    \label{eq:132}
    O:=\Big\{t\in [0,T]:\llambda^{\J+1}(V(t))\in \Int(\Lambda^{\J+1})\Big\}.
  \end{equation}
  For every $t\in O$ the set $\uU^\J(V(t))$ satisfies the minimality
  property
  \begin{subequations}
    \begin{equation}
      \label{eq:minimality}
      \text{if }\uu',\uu''\in \uU^\J(V(t))\text{ then there exists
      }\boldsymbol \nu\in \{-1,1\}^\J:\quad
      u'_j=\nu_ju''_j\quad j=1,\cdots,\J,    
    \end{equation}
    so that
    \begin{equation}
    \big\{(u_1^2,\cdots,u_J^2):
    (u_1,u_2,\cdots,u_J)\in \uU^\J(V(t)) \big\}\quad
    \text{contains a unique
  element},\label{eq:minimality2}
\end{equation}
  \end{subequations}
  and for every $\uu(t)\in
  \uU^\J(V(t))$ we have
  \begin{equation}
    \label{eq:133}
    V'(t)+\sum_{j=1}^J \partial_i\varphi(\llambda^\J(V(t)))u_i^2(t)\in - \partial_F^-\GG(V(t)).
  \end{equation}
\end{proposition}
The refined structural condition \eqref{eq:29} of Section 2.F
guarantees that the decomposition \eqref{eq:133} holds a.e.~in $(0,T)$.
\begin{theorem}
  \label{thm:main2}
  Under the same assumptions of Theorem \ref{thm:main1}, let us also
  assume that \eqref{eq:29} of Section 2.F holds, let 
  $V$ be a solution of \eqref{eq:127bis}, and let $D$ be the set
  (of full Lebesgue measure) where $V$ is differentiable and the
  inclusion \eqref{eq:127bis} holds.
  Then for every $t\in D$
  there exists $\uu(t)\in \uU^\J(V(t))$ satisfying \eqref{eq:133}.
\end{theorem}
The proof of our main results will follow from the analysis carried
out in the next Sections.
First of all, in Section~\ref{sec:regularity}, we study the regularity and the differentiability
properties of the functional $\HH$.
In Section \ref{section:F} we use these results in order to prove the
existence of discrete solutions to the Minimizing Movement scheme and
to obtain crucial structural properties of the limiting
subdifferential of $\FF$. A crucial step will also be the
Chain Rule formula in Proposition \ref{prop:chain}.

Section \ref{sec:minmov} contains the basic estimates on the
Minimizing Movement solutions.
It is not difficult to show that
a \emph{weak} Generalized Minimizing Movement exists (i.e.~a curve $V$
which is the pointwise \emph{weak} limit of a subsequence
$V_{\tau(k)}$).
The main improvement is to show that such a curve is also a
\emph{strong}
Generalized Minimizing Movement according to \eqref{eq:129}.
This fact is not obvious, since we did not assume that $\FF$ has
compact sublevels: it will be obtained by
using the compactness properties of the subdifferential of $\HH$
and a compensated compactness argument, see Proposition
\ref{primastima}.

At that point we will have all the ingredients to apply the results of
\cite{rs}:
the final discussion will be carried out in Section \ref{subsec:proofs}.

\section{Regularity and differentiability properties of eigenvalues and eigenfunctions}\label{sec:regularity}
In this Section we will always keep the structural assumptions 2.A--2.E of
Section 2;
we will explicitly mention the more refined property \eqref{eq:29} of
Section 2.F, whenever it is involved.

We will study the regularity properties of $\HH$ with
respect to $V$.
We will still denote by $\E_V:\Hilb\to (-\infty,+\infty]$ the (extended) quadratic form in
$\Hilb$ induced
by $\E_V$:
\begin{equation}
  \label{eq:15}
  \E_V(u):=
  \begin{cases}
    \E_V(u,u)&\text{if }u\in D(\E_V),\\
    +\infty&\text{if }u\in \Hilb\setminus D(\E_V).
  \end{cases}
\end{equation}
It is not difficult to check that $u\mapsto \E_V(u)+(\lambda_{\rm
  min})_-\,|u|^2$
is a convex and
lower semicontinuous functional on $\Hilb$.
\subsection{Weak continuity and Lipschitzianity}
\label{subsec:regularity}
\begin{lemma}[Weak continuity of eigenvalues and eigenfunctions]
  \label{le:conteigpot}
Let $V_n\in \K$, $n\in \N$, be a sequence weakly converging in
$\Hilb$ to $V\in
\K$ as $n\to\infty$.
For all $k,\J\in \N$ we have \[
\la_k(V_n)\rightarrow \la_k(V)\qquad \text{as }n\to+\infty
\]
and every sequence $\uu_n\in \uU^\J(V_n)$ admits a subsequence $m\mapsto
\uu_{n(m)}$ and a limit $\uu\in \uU^\J(V)$ such that
\begin{equation}
  \label{eq:19}
  u_{n(m),k}\to u_k\quad\text{strongly in }
  \dom\cap L^4(\sfD,\mm)\quad\text{for every }k\in \{1,\cdots,\J\}.
\end{equation}
\end{lemma}
The proof of Lemma~\ref{le:conteigpot} is well-known
(see e.g.~\cite[Proposition~2.5]{bgrv} for the first part of the claim
and~\cite[Proposition~3.69 and Theorem~3.71]{attouch} for the second
one
in the case of elliptic operators such as the Dirichlet Laplacian).
The $L^4(\sfD,\mm)$-convergence of eigenfunctions in
\eqref{eq:19} is a consequence of assumption \eqref{ass:main} in Section 2.D.
We provide a detailed proof of Lemma \ref{le:conteigpot}
in Appendix~\ref{sec:appmosco}.

Before stating the next corollary, we recall that $\uU^\J(V)$ denotes the collection of all the orthonormal systems of eigenfunctions associated with $\llambda^\J(V)$, see~\eqref{eq:11}.
\begin{corollary}
  \label{cor:compactnessU}
  For every $c\ge c_o$ the sets
  \begin{equation}
    \label{eq:77}
    \begin{aligned}
      \brmU^\J[c]:={}&\bigcup \Big\{\uU^\J(V):V\in \Kc c \Big\}\\
      \rmU_j[c]:={}& \Big\{u\text{ is a normalized
        $(V,\lambda_j(V))$-eigenfunction with $V\in \Kc c$}
      \Big\},       
    \end{aligned}
  \end{equation}
  are nonempty and compact in $(\V_4)^\J\subset \big(L^4(\sfD,\mm)\big)^\J$ and
  in $\V_4$ respectively.
  In particular, for every $c\ge c_o$ 
\begin{equation}
  \label{eq:56bis}
  A(c):=\sup\Big\{\|u\|_{L^4(\sfD,\mm)} :u\in \mathrm U_j[c],\ 1\le
  j\le \J\Big\}<+\infty.
\end{equation}
\end{corollary}
\begin{proof}
  It is clearly sufficient to prove the statement for $\brmU^\J[c]$.
  Thanks to Lemma~\ref{le:conteigpot}, the map
  $\lambda_\J:\K\to \R$ 
  is continuous with respect to the weak topology of $\Hilb$.
  Since $\Kc c $ is weakly compact (being
  a closed bounded convex set of $\Hilb$), we have that
  \begin{equation}
    \ell_\J(c):=\sup_{V\in \Kc c }\la_\J(V)<+\infty,\quad
    \brmU^\J[c]\subset \big(\rmU[c,\ell_\J(c)]\big)^\J,\label{eq:84}
  \end{equation}
see~\eqref{ass:main}. 
If $\uu_n\in \uU^\J(V_n)$, $n\in \N$, is a sequence
  with
  $V_n\in \Kc c $,  
we can find an increasing subsequence $k\mapsto n(k)$
and a limit $V\in \Kc c $ such that $V_{n(k)}\weakto V$ weakly in
$\Hilb$;
up to extracting a further (not relabeled)
subsequence,
Lemma \ref{le:conteigpot} shows that
$\uu_{n(k)}\to\uu\in \uU^\J(V)\subset \brmU^\J[c]$
strongly in $(\V_4)^\J$ as $k\to\infty$.
\end{proof}
We now introduce the  family of functions $\sigma_k:\K\to \Hilb$
\begin{equation}
  \label{eq:70}
  \sigma_k(V):=\sum_{h=1}^k \lambda_h(V)\quad\text{for every }V\in \K,
  \ k\in \N,
\end{equation}
which will play a crucial role
in the following, since they have a nice representation formula,
which involves orthonormal sets of cardinality $k$.
We refer to \cite{hiye,ovwo} for a more refined investigation in
finite dimension.

If $E\subset \Hilb$ is a subspace of $\Hilb$, we denote by $\Ort_k(E)$
the subset of orthonormal frames of $E^k$
\begin{equation}
  \label{eq:27}
  \Ort^k(E):=\Big\{\ww=(w_1,\cdots,w_k)\in E^k:\langle w_i,w_j\rangle =\delta_{ij}\Big\},
\end{equation}
and we have
\begin{equation}
\label{eq:71}
  \begin{aligned}
    \sigma_k(V)&=\min
    \Big\{\sum_{h=1}^k \E_V(w_h):\ww=(w_1,\cdots,w_k)\in \Ort^k(\V_4)\Big\},
  \end{aligned}
\end{equation}
where the minimum in \eqref{eq:71} is attained precisely at the
elements of $\uU^k(V)$.
An important property of the functions $\lambda_k,\sigma_k$ is their 
Lipschitzianity in $\Kc c $.

\begin{lemma}
  \label{le:Lip}
  Under the assumptions of Section 2.A--2.E, for every
  $k\in \{1,\cdots,\J\}$  and $c\ge c_o$ the functions 
  $V\mapsto \lambda_k(V)$ and $V\mapsto \sigma_k(V)$ are
  weakly continuous in $\K$ and 
  Lipschitz in $\Kc c $.
  Moreover, $\sigma_k$ is concave and
  the
  concave and globally Lipschitz function $\sigma_{k,c}:
  \Hilb\to \R$ defined by
  \begin{equation}
    \label{eq:72}
    \begin{aligned}
       \sigma_{k,c}(V):= \min \Big\{&\sum_{h=1}^k
      \E(w_h)+\int_\sfD V w_h^2\,\d\mm:
      \ww\in \brmU^\J[c]\Big\}
    \end{aligned}
  \end{equation}
  satisfies $  \sigma_{k,c}\ge \sigma_k$ on $\K$ and
  coincides with $\sigma_k$ on $\Kc c $.
\end{lemma}
\begin{proof}
  The weak continuity is a consequence of Lemma \ref{le:conteigpot};
  the regularity of $\lambda_k$ clearly follows from the analogous
  property of $\sigma_k$ since $\lambda_k=\sigma_k-\sigma_{k-1}$.
  We can thus focus on the case of $\sigma_k$.

  The fact that $\sigma_k$ is concave clearly follows from
  \eqref{eq:71},
  which represents $\sigma_k$ as a minimum of a family of bounded
  linear functionals on $\Hilb$.
  It is also clear that
  $\sigma_k\le  \sigma_{k,c}$.
  
  In order to prove that the local representation given by \eqref{eq:72}
  coincides with $\sigma_k$ if $V\in \Kc c$,
  it is sufficient to notice that
  the choice $\ww\in \uU^k(V)$ is admissible in the minimization
  \eqref{eq:72} of $ \sigma_{k,c}(V)$ (by the very definition
  \eqref{eq:77})
  so that 
\begin{equation}
  \sigma_k(V)= \sigma_{k,c}(V)
  \quad\text{for every }V\in \Kc c .
  \label{eq:22bis}
\end{equation}
We now observe that for every $\ww\in \brmU^k[c]$
the norm of the linear functionals
\begin{equation}
  \label{eq:23}
  \ell_{w_h}:V\to 
  \int_\sfD V\,w_h^2\,\d\mm,\quad
  h=1,\cdots,k,
\end{equation}
is uniformly bounded in $L^2(\sfD,\mm)$ by the constant
$A^2(c)$ given by \eqref{eq:56bis}, since 
\begin{displaymath}
\big  \|w_h^2\big\|_{L^2(\sfD,\mm)}
  \le \big\|w_h\big\|^2_{L^4(\sfD,\mm)}\le A^2(c),
\end{displaymath}
so that $\sigma_{k,c}$ satisfies
\begin{displaymath}
  \sigma_{k,c}(V)\ge \min\Big\{\sum_{h=1}^k\ell_{w_h}(V) : \ww\in \brmU^k[c]\Big\}\ge -kA^2(c)|V|
\end{displaymath}
and it is finite everywhere.
Moreover, $ \sigma_{k,c}$ is
the infimum of a family of $kA^2(c)$-Lipschitz functions on $\Hilb$
so it is $kA^2(c)$-Lipschitz as well. 
Thanks to \eqref{eq:22bis} we deduce that $\sigma_k$ is
$kA^2(c)$-Lipschitz in $\Kc c $.
\end{proof}
\subsection{Compactness properties of the limiting subdifferential of
  $\HH$.}

Let us now compute the superdifferential
of the concave functions
$\sigma_{k,c}$ defined by \eqref{eq:72};
we recall that
the Fr\'echet superdifferential $\partial^+_F\tFF$ of
a function $\tFF:\Hilb\to \R\cup\{-\infty\}$
is defined as $-\partial^-_F(-\tFF)$.
We will just write $\partial^+\tFF$ if $\tFF$ is concave.

For every $V\in \Hilb$ and $c\ge c_o$ we set
\begin{equation}
  \label{eq:85}
  \begin{aligned}
    \uU^{k,c}(V):={}&\Big\{\uu\in \brmU^k[c]: \uu \text{ is a minimizer of
    }\eqref{eq:72}\Big\},
    \\
    \Sigma_{k,c}(V):={}&\Big\{\sum_{h=1}^ku_h^2: \uu=(u_1,\cdots,u_k)\in
    \uU^{k,c}(V)\Big\},\\
    \Sigma_{k}(V):={}&\Big\{\sum_{h=1}^ku_h^2: \uu=(u_1,\cdots,u_k)\in
    \uU^{k}(V)\Big\}\quad\text{if }V\in \K.
  \end{aligned}
\end{equation}
Notice that
\begin{equation}
  \label{eq:93}
  \uU^{k,c}(V)=\uU^k(V)\quad\text{and}\quad
  \Sigma_{k,c}(V)=\Sigma_{k}(V)
  \quad\text{if $V\in \Kc c $.}
\end{equation}
\begin{lemma}
  \label{le:1}
  For every $V\in \Hilb$ and $c\ge c_o$ we have
  \begin{equation}
    \label{eq:86}
    \partial^+\sigma_{k,c}(V)=\cconv\Big(\Sigma_{k,c}(V)\Big);
  \end{equation}
  in particular if $V\in \Kc c $ 
  \begin{equation}
    \label{eq:89}
    \partial^+\sigma_{k,c}(V)=\cconv\Big(\Sigma_{k}(V)\Big).
  \end{equation}
 For $V\in \K$ we also get
  \begin{equation}
    \label{eq:134}
    \xi \in \cconv\Big(\Sigma_{k}(V)\Big)
    \quad\Rightarrow\quad
    \sigma_k(W)-\sigma_k(V)\le \langle \xi,W-V\rangle
    \quad\text{for every }W\in \K.
  \end{equation}
  Finally, $\partial^+ \sigma_{k,c}$ takes compact values
  and it is upper semicontinuous w.r.t.~the weak topology.
  $ \sigma_{k,c}$ is also Fr\'echet differentiable at every
  $V\in \Kc c $ such that $\lambda_k(V)<\lambda_{k+1}(V)$.
\end{lemma}
\begin{proof}
  We want to apply
  Lemma \ref{le:a1} in the appendix and we observe that 
  the functions $ \sigma_{k,c}$
  can be represented as in \eqref{eq:87},
  where
  \begin{equation}
    \label{eq:92}
    C:=\brmU^\J[c]\subset (\V_4)^\J,\quad
    f(\ww):=\sum_{h=1}^k w_h^2,\quad
    g(\ww):=\sum_{h=1}^k \E(w_h)\quad
   \text{for every }\ww\in \brmU^\J[c].
  \end{equation}
  We thus obtain all the properties stated for $\sigma_{k,c}$;
  notice that \eqref{eq:89} just follows by \eqref{eq:86} and \eqref{eq:93}.
  It is also worth noticing that
\begin{equation}
\text{if $V\in \Kc c $ and $\lambda_k(V)<\lambda_{k+1}(V)$ then
  $\Sigma_{k,c}(V)=\Sigma_k(V)$ is a singleton}\label{eq:94}
\end{equation}
thanks to
Corollary \ref{cor:a1}. This implies that $\sigma_{k,c}$ is
Fr\'echet differentiable at $V$ by Lemma~\ref{le:a1}.

Eventually, \eqref{eq:134} follows from \eqref{eq:89} by choosing $c$
sufficiently large so that $V\in \Kc c$ and using the fact that
$\sigma_k(V)=\sigma_{k,c}(V)$, $\sigma_k(W)\le \sigma_{k,c}(W)$.
\end{proof}
We now want to study the structure of the subdifferential of $\HH$.
We fix a constant $c\ge c_o$ and we denote by
$\varphi_c:\R^\J\to \R$ a $\rmC^1$ and Lipschitz function
whose restriction to $\Lambda^\J\cap [\lambda_{\rm
  min},1+\ell_\J(c)]^\J$ coincides with $\varphi$
(recall \eqref{eq:84}).
We introduce the function $\psi_c\in \rmC^1(\R^\J)$
\begin{equation}
  \label{eq:78}
  \psi_c(s_1,s_2,\cdots,s_\J):=\varphi_c(s_1,s_2-s_1,s_3-s_2,\cdots,s_\J-s_{\J-1})
\end{equation}
which clearly satisfies
\begin{equation}
  \label{eq:96}
  \partial_j\psi_c=\partial_j\varphi_c-\partial_{j+1}\varphi_c\quad  \text{if
  }1\le j<\J,\quad
  \partial_\J\psi_c=\partial_\J\varphi_c,
\end{equation}
and we set
\begin{equation}
  \label{eq:95}
   \HH_c(V):=\psi_c(\ssigma_c(V)),\quad
  \ssigma_c(V):=( \sigma_{1,c}(V),
  \sigma_{2,c}(V),\cdots,
   \sigma_{\J,c}(V))\quad\text{for every }V\in \Hilb.
\end{equation}
It turns out that $\HH_c$ is a weakly continuous and strongly Lipschitz function which
coincides with $\HH$ on $\Kc c$.
Calling the map \[
{\mathbf \Lambda}^\J\colon \Lambda^\J\to \R^\J,\qquad {\mathbf \Lambda}^\J(\lambda_1,\dots,\lambda_\J)=(\lambda_1,\lambda_1+\lambda_2,\dots, \lambda_1+\dots+\lambda_\J),
\]
we define $\psi$ as the restriction of $\psi_c$ to ${\mathbf \Lambda}^\J\Big(\Lambda^\J\cap [\lambda_{\rm
  min},1+\ell_\J(c)]^\J\Big)$. In particular, \[
  \psi(s_1,s_2,\cdots,s_\J)=\varphi(s_1,s_2-s_1,\cdots,s_\J-s_{\J-1}), \qquad \text{for }(s_1,s_2,\cdots,s_\J)\in {\mathbf \Lambda}^\J\Big(\Lambda^\J\cap [\lambda_{\rm
  min},1+\ell_\J(c)]^\J\Big).
\]
Let us now introduce the multivalued map
$\subH_c:\Hilb\to 2^{\Hilb}$,
\begin{equation}
  \label{eq:93b}
  \subH_c(V):=\sum_{j=1}^\J
  \partial_j\psi_c(\ssigma_c(V))\partial^+\sigma_{j,c}(V)=
  \Big\{\sum_{j=1}^\J
\gamma_j \xi_j:
    \gamma_j=\partial_j\psi_c(\ssigma_c(V)),\
    \xi_j\in \partial^+\sigma_{j,c}(V)\Big\}.
  \end{equation}
  When $V\in \Kc c$, $\subH_c(V)$ is independent of $c$ and can be
  written as
  \begin{equation}
    \label{eq:108}
    \subH_c(V)=\subH(V)=
    \Big\{\sum_{j=1}^\J
    \gamma_j \xi_j:
    \gamma_j=\partial_j\psi(\ssigma(V)),\
    \xi_j\in \cconv\Big(\Sigma_j(V)\Big)\Big\}.
  \end{equation}

  \begin{proposition}[Compactness of the limiting subdifferential of $\HH_c$]
    \label{prop:compactness}
  Let $c\ge c_o$ be given.
  \begin{enumerate}
  \item For every weakly compact set $B\subset \Hilb$
    (in particular for $B=\Kc c$) the set
    \begin{equation}
      \label{eq:101}
      \bigcup_{V\in B}\subH_c(V)\quad\text{is strongly compact in }\Hilb
    \end{equation}
    and the graph of $\subH_c$ is weakly closed in $\Hilb\times \Hilb$:
    \begin{equation}
      \label{eq:97}
      (V_n,\xi_n)\weakto (V,\xi)\quad \xi_n\in
      \subH_c(V_n)\quad\Rightarrow
      \quad
      \xi\in \subH_c(V).
    \end{equation}
  \item
    For every $V\in \Hilb$, 
    $\partial_L^- \HH_c(V)$ is not empty and
    $\partial_L^- \HH_c(V)\subset \subH_c(V)$.
  \end{enumerate}
\end{proposition}
\begin{proof}
  Claim (1) is an easy consequence of Lemma \ref{le:1}, the
  representation \eqref{eq:92} in terms of Lemma \ref{le:a1}, 
  and the fact that $\psi_c$ is of class
  $\rmC^1$.

  Claim (2) follows by the application of the calculus properties of limiting
  subdifferentials
  of Lipschitz functions: the chain rule
  \cite[Ch.~1, Thm.~10.4]{Clarke}, the sum
  rule
  \cite[Ch.~1, Prop.~10.1]{Clarke}, and the fact that $\partial_L^- f\subset
  \partial^+f \cup\partial^- f$ for convex or concave functions.
\end{proof}
\subsection{Superdifferentiability}
We want now to show that if $\varphi$ satisfies the further structural
conditions
stated in Section 2.F, then $\HH$ has a nice superdifferentiability
property in $\Kc c$.

\begin{theorem}[Superdifferentiability of $\HH$]
  \label{thm:superdiff}
Let $c\ge c_o$ and let $\HH$ satisfy the structural assumptions
\eqref{eq:29} of Section 2.F.
If $V\in \Kc c$, $\uu=(u_1,\cdots,u_\J)\in \uU^\J(V)$
and $\xi=\sum_{i=1}^\J\partial_i\varphi(\llambda^\J(V))u_i^2$, then 
$\xi\in \partial^+_F \HH_c(V)$; in particular
there exists a positive function $\omega:\Hilb \to \R$
as in \eqref{eq:barrier} and $\varrho>0$ such that 
    \begin{equation}
      \HH(W)-\HH(V)-\langle \xi,W-V\rangle \le\omega(W-V)
      \quad\text{for every }W\in \rmB(V,\varrho)\cap \Kc c.
      \label{eq:44bis}
    \end{equation}
\end{theorem}
\begin{proof}
  Let us recall that $\ssigma_c(V)=\ssigma(V)$ since $V\in \Kc c$;
  we set
  \begin{displaymath}
    \lambda_i:=\lambda_i(V),\quad
    p_i:=\partial_i\varphi(\llambda^\J(V)),\quad
    \gamma_i:=\partial_i\psi(\ssigma(V))=p_i-p_{i+1}.
\end{displaymath}
  The differentiability of $\psi_c$ and the fact that $ W\mapsto
  \sigma_{i,c}(W)$ is Lipschitz entail that
  \begin{equation}
    \label{eq:48}
    \psi_c( \ssigma_c(W))-
    \psi_c(\ssigma_c(V))=
    \sum_{i=1}^\J \gamma_i(\sigma_{i,c}(W)-\sigma_{i,c}(V))+
    o(|W-V|)\quad\text{as }W\to V.
  \end{equation}
  Let us consider the set $H$ of indices $\{j:1\le j<\J,\ 
  \lambda_j<\lambda_{j+1}\}$
  and observe that $\gamma_j\ge0$ if $j\not\in H$
  thanks to \eqref{eq:29}.

  By Lemma \ref{le:1} $\sigma_{i,c}$ is Fr\'echet differentiable
  at $V$ for every $i\in H$ and it is Fr\'echet superdifferentiable
  for every $i$. It follows that
  setting
  $\xi_i:=\sum_{k=1}^iu_k^2$, 
  \begin{equation}
    \label{eq:103}
    \gamma_i \xi_i \quad\text{belongs to the Fr\'echet
      superdifferential of}\quad
    W\mapsto \gamma_i\sigma_{i,c}(W)\quad\text{at $V$}.
  \end{equation}
  Using \eqref{eq:48} we find a positive
  function $\omega:\Hilb\to \R$ as
  in \eqref{eq:barrier} and $\varrho>0$ such that 
   \begin{equation}
    \label{eq:48bis}
    \psi_c( \ssigma_c(W))-
    \psi_c(\ssigma_c(V))\le 
    \langle \sum_{i=1}^\J \gamma_i \xi_i, W-V\rangle+
    \omega(W-V)
    \quad\text{for every }W\in \rmB(V,\varrho).
  \end{equation}
  On the other hand
  \begin{equation}
    \label{eq:141}
    \sum_{i=1}^\J \gamma_i \xi_i=p_\J\xi_\J+
    \sum_{i=1}^{\J-1} (p_i-p_{i+1}) \xi_i=
    \sum_{i=2}^{\J} p_i(\xi_i-\xi_{i-1})+p_1\xi_1
    =\sum_{j=1}^\J p_i u_i^2=\xi.
  \end{equation}
Inequality~\eqref{eq:44bis} then follows by \eqref{eq:48bis} and the fact that
  $\HH_c(V)=\HH(V)$ and $\HH_c(W)=\HH(W)$ if $V,W\in \Kc c$.
\end{proof}
\subsection{The case when $\HH$ is concave}
\label{subsec:func}
The result of the previous section can be further refined when
$\varphi$ satisfies the stronger condition,
\begin{equation}
  \label{eq:40bis}
  \partial_{i}\varphi\ge
  \partial_{i+1}\varphi\quad\text{for every }1\le i<\J,\quad
  \partial_\J \varphi\ge 0
  \quad\text{in
  }\Lambda^\J,
\end{equation}
which is related to Schur-concavity \cite{moa}.
Even though the superdifferentiability result, Theorem \ref{thm:superdiff}, covers a
more general setting, let us briefly recap this different approach.

We consider here the situation when $\varphi$ is the restriction
to
$\Lambda^\J$ of a $\rmC^1$ symmetric function $\phi:[\lambda_{\rm
  min},+\infty)^\J\to \R$
(recall \eqref{eq:30}).
We consider the functions $S_k:\R^\J\to \R$, $1\le k\le
\J$, defined by 
\begin{equation}
  \label{eq:34}
  S_k(\mmu)=S_k(\mmu_\uparrow):=\sum_{h=1}^k\mu_{(h)}
\end{equation}
where for every $\boldsymbol \mu=(\mu_1,\cdots,\mu_\J)\in \R^\J$
we will denote by $\boldsymbol
\mu_\uparrow=(\mu_{(1)},\cdots,\mu_{(\J)})\in \Lambda^\J$
the vector obtained by increasing rearrangement of the component of
$\boldsymbol\mu$.

The functions $S_k$ induce a partial order on $\R^\J$ given by
\begin{equation}
  \label{eq:35}
  \mmu'\prec \mmu''\quad\text{if and only if }\quad
  S_k(\mmu')\ge S_k(\mmu'')\quad\text{for every }k=1,\cdots,\J;
\end{equation}
if \eqref{eq:35} holds we say that $\mmu'$ is weakly majorized by
$\mmu''$.

If $E\subset \Hilb$ is a subspace of dimension $d\ge \J$
and $V\in \Hilb$, we can consider the vector
$\llambda^\J(V,E)=(\lambda_1(V,E),\cdots,
\lambda_\J(V,E))$ of the eigenvalues of the restriction of $\E_V$ to
$E$.
The variational characterization easily shows that
\begin{equation}
  \label{eq:36}
  \lambda_k(V)\le \lambda_k(V,E)\quad\text{for every }1\le k\le \J,
\end{equation}
so that in particular $\llambda^\J(V,E)\prec \llambda^\J(V)$.
By a Theorem of Schur \cite[Chap.~9, B1]{moa}, if
$\ww=(w_1,\cdots,w_\J)\in \Ort^\J(E)$
and $\mmu=(\mu_1,\cdots,\mu_\J)$ with $\mu_k=\E_V(w_k)$ we also have
\begin{equation}
  \label{eq:37}
  \mmu=\E_V(\ww)\prec \llambda^\J(V,E).
\end{equation}
By selecting $E=\operatorname{Span}(\ww)$ we conclude that
\begin{equation}
  \label{eq:38}
 \E_V(\ww)\prec \llambda^\J(V)\quad\text{for every }\ww\in \Ort^\J(\Hilb).
\end{equation}
If $\phi\in \rmC^1\big([\lambda_{\rm min},+\infty)^\J\big)$ is
symmetric
then it satisfies the monotonicity condition
\begin{equation}
  \label{eq:39}
  \llambda'\prec \llambda''\quad\Rightarrow\quad \phi(\llambda')\ge \phi(\llambda'')
\end{equation}
if and only if $\phi$ is increasing and Schur-concave
\cite[Chap.~3, A8]{moa}, i.e.
\begin{equation}
  \label{eq:40}
  (\partial_{1}\phi(\llambda)-
  \partial_{2}\phi(\llambda))(\lambda_1-\lambda_2)\le 0
  \quad \text{for every }\llambda\in [\lambda_{\rm min},+\infty)^\J,
\end{equation}
which implies \eqref{eq:40bis} thanks to the symmetry of $\phi$.

It is worth noticing that this class contains all the concave
increasing functions, so that 
\begin{equation}
  \label{eq:41}
  \text{if $\varphi$ is induced by a symmetric, increasing and concave
    function $\phi$, then
    \eqref{eq:39} holds}.
\end{equation}
In particular $S_\J$ satisfies \eqref{eq:39}.
However, the class of symmetric Schur-concave functions
is much wider and stable w.r.t.~various kind of operations, see
\cite{moa}.
In particular
\begin{quote}
  all the elementary symmetric polynomials are Schur-concave 
\end{quote}
and
  increasing if $\lambda_{\rm min}\ge0$.
We deduce the following result.
\begin{proposition}[Concavity of $\HH$]
  \label{phiminlemma}
  Let $k\in \N$, $V\in \K$, $\phi\in \rmC^1([\lambda_{\rm
    min},+\infty)^k)$ be
  a symmetric, increasing and Schur-concave function.
  Then
\begin{equation}\label{phimin}
  \HH(V)=\min{\left\{\phi(\E_V(\ww))\;:\ww\in
    \big(\dom\cap L^4(\sfD,\mm)\big)^k\cap \Ort^k(\Hilb)\right\}}.
\end{equation}
If moreover $\phi$ is concave, then
the function $\HH$ is concave as well.
\end{proposition}
\begin{proof}
  \eqref{eq:38} and \eqref{eq:39} yield
  \begin{equation}
    \label{eq:42}
    \phi(\llambda^k(V))\le \phi(\E_V(\ww))\quad
    \text{for every }\ww\in
    \big(\dom\cap L^4(\sfD,\mm)\big)^k\cap \Ort^k(\Hilb).
  \end{equation}
  On the other hand, the equality is attained by selecting $\ww\in
  \uU^k(V)$.

  When $\phi$ is concave the maps
  $V\mapsto \phi(\E_V(\ww)) $ are concave
  since they are the composition of a concave with a linear function
  w.r.t.~$V$. It follows that $V\mapsto \phi(\llambda^k(V))$ is
  concave
  as well, since it is the minimum of a family of concave functions.
\end{proof}

\section{Regularity and subdifferentiability properties of $\FF$}
\label{section:F}
In this section we will collect the main properties of the functional $\FF$ from~\eqref{eq:13},
according to the setting presented in Section 2.A--2.E.
We will eventually discuss a further important consequence of 
\eqref{eq:29} from Section 2.F.
\begin{lemma}[Weak continuity and coercivity of $\FF$]
  \label{le:Flsc}
  For every $\eta>\theta$ the function
  $\FF_\eta:=\FF+\frac\eta2|\cdot|^2$
  is weakly lower semicontinuous and
  there exists a constant $S(\eta)\ge 0$ such that
  \begin{equation}
    \label{eq:107}
    \FF_\eta(V)\ge \delta|V|^2-S(\eta)\quad\text{for every }V\in
    \Hilb,\quad
    \delta:=(\eta-\theta)/6;
  \end{equation}
  in particular the sublevels of $\FF_\eta$ are bounded (thus
  weakly compact) in $\Hilb$.\\
  For every $a\ge 0$ there exists $c=c(a)>0$ such that
  if $|V|\le a$ and $\FF(V)\le a$ then $V\in \Kc c$.\\  
  In particular for every $\tau>0$ such that $\tau\theta<1$
  and every $V\in \Hilb$ the
  functional $\Phi(\tau,V;\cdot):\Hilb\to \R\cup\{+\infty\}$ 
  \begin{equation}
    \label{eq:105}
    \Phi(\tau,V;W):= \frac{1}{2\tau}|W-V|^2+\FF(W),\quad
    W\in \Hilb
  \end{equation}
  has a minimizer.  
\end{lemma}
\begin{proof}
  By \eqref{eq:65} and the fact that
  \begin{displaymath}
    (\lambda_\J)_+\le (\lambda_{\rm min})_++
    \sum_{j=1}^\J (\lambda_j-\lambda_{\rm min})
    \le (\lambda_{\rm min})_++ \J (\lambda_{\rm min})_-+\sum_{j=1}^\J \lambda_j,\qquad \llambda\in \Lambda^\J,
  \end{displaymath}
  we obtain
  \begin{displaymath}
    \HH(V)\ge -A_1(1+\sigma_\J(V))\ge
    -A_1(1+ \sigma_{\J,c_o}(V))
    \quad\text{for every $V\in \K$}
  \end{displaymath}
  with $A_1:=A(1+\J |\lambda_{\rm min}|)$.
  Since the function
  $V\mapsto -A_1\sigma_{\J,c_o}(V)$ is convex, finite, and 
  continuous in $\K $
  thanks to the representation \eqref{eq:72},
  it is bounded from below by an affine function, so that
  there exists a constant $A_2>0$ such that
  \begin{equation}
    \label{eq:75}
    \HH(V)\ge -A_2(1+|V|)\quad\text{for every }V\in \Hilb.
  \end{equation}
  Setting $\delta:=(\eta-\theta)/6$ and 
  $A_3:=A_2+A_2^2/4\delta$
  we get
  \begin{equation}\label{eq:80}
    \FF_\eta(V)\ge -A_2(1+|V|)+\GG_\theta(V)+
                 3\delta|V|^2\ge -A_3+\GG_\theta(V)+2\delta|V|^2,
  \end{equation}
  showing that every sublevel of $\FF_\eta$ is contained in a suitable
  sublevel of $\GG_\theta$. 
  Since $\GG_\theta$ is convex, we have for some $A_4\ge0$
  \begin{equation}
    \label{eq:79}
    \GG_\theta(V)\ge -A_4(1+|V|) \quad\text{for every }V\in \Hilb,
  \end{equation}
  so that \eqref{eq:80} yields for $A_5:=A_3+A_4$ and
  $A_6:= A_5+A_5^2/4\delta$
  \begin{equation}
    \label{eq:106}
    \FF_\eta(V)\ge -A_5(1+|V|)+2\delta|V|^2
    \ge -A_6+\delta|V|^2,
  \end{equation}
  showing \eqref{eq:107}.
  In particular if $\FF(V)\le a$ and $|V|\le a$ then \eqref{eq:80}
  shows that $V\in \Kc c $
  whenever $c\ge a+\frac12\eta ^2a+A_3$.

  Since the restriction of $\HH$ to $\Kc c$ is weakly continuous
  and $\GG_\eta$ is convex and weakly lower semicontinuous as well, we
  conclude
  that $\FF_\eta$ is also weakly lower semicontinuous.
  Since
  \begin{displaymath}
    \Phi(\tau,V;W)=
    \frac1{2\tau}|V|^2-\frac1\tau\langle W,V\rangle+\FF_{\tau^{-1}}(W),
  \end{displaymath}
  if $\tau^{-1}>\theta$ we immediately get that $\Phi(\tau,V;\cdot)$
  has a minimizer.
\end{proof}
Let us now study the properties of the limiting subdifferential of
$\FF$.
We will also consider a weaker notion
of $\ell$-subdifferential: we say that 
$\xi$ belongs to $\partial_\ell^-\FF(v)$
if there exist sequences $v_n,\xi_n\in \Hilb$ such that 
\begin{equation}\label{eq:limitsubdiffweak}
  \xi_n\in \partial_F^-\FF(v_n),\quad v_n\rightarrow v\;\text{strongly
    in $\Hilb$},\quad \xi_n\rightharpoonup \xi\;\text{weakly in
  $\Hilb$},\quad \sup_{n}\FF(v_n)<\infty,
\end{equation}
see also Remark~\ref{rem:tedious}.

\begin{lemma}[Decomposition of the limiting subdifferential of $\FF$ -
  I]
  \label{Elk1}
  For every $V\in \K$ we have $\partial_\ell^-\FF(V)=\partial_L^-\FF(V)$.\\
  If $V\in \Kc c$, $\xi\in \partial^-_L\FF(V)$, and 
  $c_1>c$ then
  there exist
  $\xi_H\in \partial_L^- \HH_{c_1}(V)$
  and
  $\xi_K\in \partial^-_F \GG(V)$ such that
  $\xi=\xi_H+\xi_K$.
  In particular
  there exist $\xi_j\in \cconv\Big(\Sigma_j(V)\Big)$
  such that
  \begin{equation}
    \label{eq:109}
    \xi=\sum_{j=1}^J\gamma_j\xi_j+\xi_K,\quad
    \gamma_j=\partial_j\psi_{c_1}(\ssigma(V)).
  \end{equation}
\end{lemma}
\begin{proof}
  We set $a:=c_1-c$
  and we first consider the case when $\xi\in \partial^-_F\FF(V)$
  is an element of the Fr\'echet subdifferential of $\FF$.

  In this case there exists $\varrho>0$
  and a positive function $\omega:\Hilb\to\R$
  as in \eqref{eq:barrier} such that 
    \begin{equation}
    \label{eq:60bis}
    \HH(W)-\HH(V)+\GG(W)-\GG(V)-\langle \xi,W-V\rangle\ge
    -\omega(W-V)\quad
    \text{for every }W\in \rmB(V,\varrho).
  \end{equation}
  If $\delta<a$, $W\in \rmB(V,\delta)$, and $W\not\in \Kc{c_1}$ then
    \begin{align*}
    \GG(W)-\GG(V)
    &=
    \GG(W)+\frac \theta2|W|^2-\big(\GG(V)+
    \frac\theta 2|V|^2\big)-
    \frac\theta 2\big(|W|^2-|V|^2\big)
    \\&\ge
      a-\frac\theta 2(|W|+|V|)|W-V|
      \ge a-\theta c_1\delta,
    \end{align*}
    since $|W|\le |V|+\delta\le c_1$ and $|V|\le c\le c_1$.
    On the other hand, if $L$ is the Lipschitz constant of $
    \HH_{c_1}$
    we get
    \begin{equation}
      \label{eq:110}
      \HH_{c_1}(W)-\HH_{c_1}(V)+\GG(W)-\GG(V)-\langle
      \xi,W-V\rangle
      \ge a-(L+\theta c_1+|\xi|)\delta\ge 0,
    \end{equation}
    if we choose $\delta>0$ so small that
    $(L+\theta c_1+|\xi|)\delta<a$.
    Possibly replacing $\varrho$ with $\delta$, since $ \HH_{c_1}$
    concides with $\HH$ on $\Kc {c_1}$ we deduce from \eqref{eq:60bis}
    and \eqref{eq:110} that
    $\xi\in \partial_F^-(\HH_{c_1}+\GG)(V)$.
    We can then apply the sum rule for the limiting subdifferential
    \cite{Clarke} and we 
    obtain the decomposition
    \begin{equation}
      \label{eq:111}
      \xi=\xi_H+\xi_K,\quad
      \xi_H\in \partial^-_L \HH_{c_1}(V),\quad
      \xi_K\in \partial^-_F\GG(V).
    \end{equation}
    Let us now consider the general case when
    $\xi\in \partial_\ell\FF(V)$. 
    By \eqref{eq:limitsubdiffweak}, we can find
    $V_n\in \K$ and $\xi_n\in \partial_F^-\FF(V_n)$ such that
    \begin{equation}
      \label{eq:112}
    V_n\to V,\quad
    \xi_n\weakto \xi,\quad
    \sup_n\FF(V_n)\le C<+\infty.
  \end{equation}
  We thus find a suitably large constant $c_2$ such that $V_n\in \Kc
  {c_2-1}$ and therefore we can decompose
  $\xi_n$ as
  \begin{equation}
    \label{eq:113}
    \xi_n=\xi_H^n+\xi_K^n,\quad
    \xi_H^n\in \partial^-_L \HH_{c_2}(V_n),\quad
    \xi_K^n\in \partial^-_F\GG(V_n).
  \end{equation}
  It follows that $\xi_H^n$ is uniformly bounded, so that also
  $\xi_K^n$ is uniformly bounded.
  Since $\GG$ is $(-\theta)$-convex, this implies that
  $\GG(V_n)\to \GG(V)$; on the other hand
  $\HH$ is continuous in $\K$ so that
  $\HH(V_n)\to \HH(V)$ as well, showing that
  $\xi\in \partial_L^-\FF(V)$.
  
  Choosing $c'\in (c,c_1)$ we definitely have $V_n\in \Kc {c'}$.
  We can thus refine the decomposition \eqref{eq:113}
  and assume that $\xi_H^n\in \partial^-_L
  \HH_{c_1}(V_n)\subset \subH_{c_1}(V_n)$.
  We can now extract an increasing subsequence
  $k\mapsto n(k)$ such that
  $\xi^{n(k)}_H\to \xi_H$ for some $\xi_H\in \partial^-_L
  \HH_{c_1}(V_n)$
  (here we use the closedness of the limiting subdifferential)
  and therefore $\xi^{n(k)}_K\weakto \xi-\xi_H\in \partial_F^-\GG(V)$.

  \eqref{eq:109} then follows by Proposition \ref{prop:compactness}
  and \eqref{eq:108}.
\end{proof}
\begin{corollary}
  \label{cor:vitafacile}
  Under the same assumption of Lemma \ref{Elk1}, let us suppose
  that $\llambda^{\J+1}(V)\in \Int(\Lambda^{\J+1})$ so that
  $\uU^\J(V)$ satisfies the minimality properties
  \eqref{eq:minimality}--\eqref{eq:minimality2}.
  If $\xi\in \partial_L^-\FF(V)$ and $\uu\in \uU^J(V)$, then we have
    \begin{equation}
      \label{eq:59tris}
      \xi-\sum_{i=1}^\J\partial_i\varphi(\llambda^\J(V))u_i^2\in \partial_F^- \GG(V).
    \end{equation}
  \end{corollary}
  \begin{proof}
    If $\llambda^{\J+1}(V)\in \Int(\Lambda^{\J+1})$ then (see
    \eqref{eq:131})
    $\lambda_1(V)<\lambda_2(V)<\cdots<\lambda_\J(V)<\lambda_{\J+1}(V)$
    so that the set of normalized $(V,\lambda_j(V))$ eigenfunctions
    contains only two (opposite) elements for $1\le j\le \J$ and
    \eqref{eq:minimality}--\eqref{eq:minimality2} hold.

    If $\uu\in \uU^\J(V)$
    we have $\Sigma_k(V)=\{\xi_k\}$ where $\xi_k=\sum_{j=1}^ku_j^2$
    for every $k\in \{1,\cdots,\J\}$.
    Using \eqref{eq:109}, 
    we  can then argue as in \eqref{eq:141} to obtain
    \begin{displaymath}
      \sum_{j=1}^\J \gamma_j\xi_j=
      \sum_{i=1}^\J\partial_i\varphi(\llambda^\J(V))u_i^2.\qedhere
    \end{displaymath}
  \end{proof}
As a further step, we will prove that
the limiting subdifferential of $\FF$ contains sufficient information
to get the following chain rule property
(cf.~condition \textsc{(chain$_2$)} of \cite[Thm.~3]{rs}).
\begin{proposition}[Chain rule]
  \label{prop:chain}
  Let $V\in H^1(0,T;\Hilb)$, $\xi\in L^2(0,T;\Hilb)$ such that
  $\xi(t)\in \partial_L^-\FF(V(t))$ for a.e.~$t\in (0,T)$ and
  $\FF\circ V$ is bounded.
  Then the map $\FF\circ V$ is absolutely continuous in $[0,T]$ and
  \begin{equation}
    \label{eq:114}
    \frac\d{\d t}\FF(V(t))=\langle
    \xi(t),V'(t)\rangle\quad\text{a.e.~in $(0,T)$}.
  \end{equation}
\end{proposition}
\begin{proof}
  Since $\FF\circ V$ is bounded and $V$ is bounded as well in $\Hilb$
  being $V\in H^1(0,T;\Hilb)$,
  by Lemma \ref{le:Flsc} there exists a constant $c\ge c_o+1$ such that
  $V(t)\in \Kc {c-1}$ for every $t\in [0,T]$.

  We deduce that $\FF\circ V= \FF_{c}\circ V$ where
  $\FF_c=\GG+\HH_c$.
  Since $ \HH_c$ is a Lipschitz function, the composition
  $t\mapsto \HH_c\circ V(t)$ is absolutely continuous.
  Moreover by Lemma \ref{Elk1} we can decompose $\xi(t)$ as
  \begin{equation}
    \label{eq:115}
    \xi(t)=\xi_H(t)+\xi_K(t),\quad
    \xi_H\in \partial_L^-\HH_c(V(t)),\quad
    \xi_K(t)\in \partial_F^-\GG(V(t))\quad
    \text{for a.e.~$t\in (0,T)$}.
  \end{equation}
  Since $\HH_c$ is Lipschitz, $\xi_H$ is uniformly bounded and
  therefore
  the minimal selection $t\mapsto \partial_F^\circ\GG(V(t))$ is a
  function in $L^2(0,T;\Hilb)$.
  Being $\GG$ the difference between a convex function and a quadratic
  one, we conclude that
  $t\mapsto \GG(V(t))$ is absolutely continuous as well.

  We can then find a Borel set $D\subset (0,T)$ of full Lebesgue
  measure such that
  the functions $V,\ \HH_c\circ V,\ \GG\circ V,\ \sigma_{j,c}\circ
  V$ are differentiable at every $t\in D$, $j=1,\cdots,\J$, and there exist
  $\xi_j(t)\in \partial^+\sigma_{j,c}(V(t))$ and $\xi_K(t)\in
  \partial_F^-\GG(V(t))$ such that
  \begin{equation}
    \label{eq:116}
    \xi(t)=\sum_{j=1}^J\gamma_j(t)\xi_j(t)+\xi_K(t),\quad
    \gamma_j(t)=\partial_j\psi_c(\ssigma_{k,c}(V))\quad
    \text{for every }t\in D,
  \end{equation}
  thanks to \eqref{eq:109}.
  Since $\sigma_{j,c}$ are concave and $\GG$ is $(-\theta)$-convex,
  we have
  \begin{equation}
    \label{eq:117}
    \frac\d{\d t}\sigma_{j,c}(V(t))=
    \langle\xi_j(t),V'(t)\rangle,\quad
    \frac\d{\d t}\GG(V(t))=\langle \xi_K(t),V'(t)\rangle
    \quad\text{for every }t\in D.
  \end{equation}
  Since $\psi_c$ is of class $\rmC^1$ we clearly have
  \begin{equation}
    \label{eq:118}
    \frac\d{\d t}\HH_c(V(t))=
    \frac\d{\d t}\psi_c(\ssigma_c(V(t))=
    \sum_{j=1}^\J \gamma_j (t) \frac\d{\d t }\sigma_{j,c}(V(t))
    \quad\text{for every }t\in D.
  \end{equation}
  Combining \eqref{eq:118}, \eqref{eq:117} and \eqref{eq:116} we get
  \eqref{eq:114}.
\end{proof}
We conclude this section by showing a more refined decomposition of
$\partial_L^-\FF$ in the case $\varphi$ satisfies also the
structural condition \eqref{eq:29} of Section 2.F.
\begin{lemma}[Decomposition of the subdifferential of $\FF$ - II]
  \label{ELk}
  Let us suppose that all the assumptions of Section 2 are satisfied,
  including \eqref{eq:29}.
  \begin{enumerate}
  \item
    If $V\in \K$ and $\xi\in \partial^-_F\FF(V)$ then for every
    $\uu\in \uU^\J(V)$ 
    \begin{equation}
      \label{eq:59}
      \xi-\sum_{i=1}^\J\partial_i\varphi(\llambda^\J(V))u_i^2\in \partial_F^-\GG(V).
    \end{equation}
  \item
    If $V\in \K$ and $\xi\in \partial_L^-\FF(V)$ then there
    exist $\uu\in \uU^\J(V)$ and $\xi_K\in \partial^-_F\GG(V)$ such
    that
    \begin{equation}
      \label{eq:59bis}
      \xi=\sum_{i=1}^\J\partial_i\varphi(\llambda^\J(V))u_i^2+ \xi_K.
    \end{equation}
  \end{enumerate}
\end{lemma}
\begin{proof}
  Let us first consider Claim (1).
  By Definition \ref{subdiff} we know that
  there exist $\varrho>0$ and
  a positive function $\omega_\FF:\Hilb\to\R$ as in \eqref{eq:barrier}
  such that 
  \[
\HH(W)-\HH(V)+\GG(W)-\GG(V)-\langle \xi,W-V\rangle\ge -\omega_\FF(W-V)
\]
for every $W\in \K\cap\rmB(V,\varrho).$
  Let us set $\xi_H:=\sum_{i=1}^\J\partial_i\varphi(\llambda^\J(V))u_i^2
  $
  for some $\uu\in \uU^\J(V)$ 
  and let us select $c\ge c_o$ such that\[
    c\ge |V|+1, \qquad \text{and}\qquad  c\ge \GG(V)+\frac 12\theta|V|^2+1
\] 
so that $V\in\Kc {c-1}$.

We can now apply 
\eqref{eq:44bis} and find $\varrho_1\in (0,\varrho)$
and a positive function $\omega_\HH:\Hilb\to \R$ as in \eqref{eq:barrier}
such that
  \begin{equation}
    \label{eq:61}
    \HH(W)-\HH(V)\le \langle
    \xi_H,W-V\rangle+\omega_\HH(W-V)\quad
    \text{for every }W\in \rmB(V,\varrho_1)\cap \Kc c
  \end{equation}
  so that
  \begin{equation}
    \label{eq:62bis}
    \GG(W)-\GG(V)-\langle \xi-\xi_H,W-V\rangle
    \ge - \Big(\omega_\HH(W-V)+\omega_\FF(W-V)\Big),    
  \end{equation}
  for every $W\in \Kc c\cap\rmB(V,\varrho_1)$.
  
  On the other hand, if $W\not \in \Kc c$
  and $|W-V|\le\delta$ we have
  \begin{align*}
    \GG(W)-\GG(V)
    &=
    \GG(W)+\frac \theta2|W|^2-\big(\GG(V)+
    \frac\theta 2|V|^2\big)-
    \frac\theta 2\big(|W|^2-|V|^2\big)
    \\&\ge
    1-\frac\theta 2(|W|+|V|)\delta,
  \end{align*}
  so that choosing $\delta<\varrho_1$ sufficiently small we obtain
  \begin{equation}
    \label{eq:63}
    \GG(W)-\GG(V)-\langle \xi-\xi_H,W-V\rangle\ge \delta/2\quad\text{if
      $W\not\in \Kc c$ and $|W-V|<\delta$.}
  \end{equation}
  This implies that
  \eqref{eq:62bis}
  holds for every $W\in \K\cap\rmB(V,\delta)$
  and therefore 
  $\xi-\xi_H\in \partial^-_F\GG(V)$.

  Claim (2) readily follows: by the definition of limiting
  subdifferential we can find
  a sequence $V_n\in \K$ strongly convergent to $V$ and
  $\xi_n\in \partial_F^-\FF(V_n)$ weakly convergent to $\xi$ with
  $\FF(V_n)\to \FF(V)$ as $n\to\infty$. 
  We can then select arbitrary $\uu_n\in \uU^\J(V_n)$ setting
  $\xi_H^n:=\sum_{i=1}^\J\partial_i\varphi(\llambda^\J(V_n))(u_i^n)^2$
  and $\xi_K^n:=\xi_n-\xi_H^n\in \partial_F^-\GG(V_n)$.

  Up to extracting a subsequence, we may assume that
  $\uu^n\to\uu\in \uU^\J(V)$ strongly in $\Hilb^\J$
  and $\llambda^\J(V_n)\to \llambda^\J(V)$ in $\Lambda^\J$, so that
  $\xi_H^n\to
  \xi_H:=\sum_{i=1}^\J\partial_i\varphi(\llambda^\J(V))u_i^2$
  strongly in $\Hilb$ thanks to the regularity of $\varphi$.
  Correspondingly we have
  $\xi_k^n\weakto \xi_K=\xi-\xi_H\in \partial_F^-\GG(V)$
  since $\partial_F^-\GG$ is strongly-weakly closed.
\end{proof}

\section{Convergence of the Minimizing Movement scheme and proof of
  the main results}\label{sec:minmov}

We now refer to the construction we introduced in Section
\ref{sec:main} (see in particular \eqref{eq:104},
\eqref{impleuler}, \eqref{eq:interpolantlpc} and Definition
\ref{def:GMM}) and we
briefly recap the main
general properties and estimates 
from the abstract theory of Minimizing Movements,
following \cite[Section~4]{rs}.
As usual, we operate in the setting of Section
\ref{sec:settingresults},
2.A--2.E.

\subsection{Existence, stability estimates and weak convergence of Generalized Minimizing Movements}

We start by proving the existence of Generalized Minimizing Movements in our setting. 
\begin{lemma}\label{le:h1h2h3ok}
  Let $\tau_*>0$ such that $\theta\tau_*<1$.
  \begin{enumerate}
  \item 
    For every $\tau\in (0,\tau_*)$ and $V_0\in \K$
    there exists a discrete solution $(V^n_\tau)_{0\le n\le N(\tau)}$
    to the Minimizing Movement scheme \eqref{impleuler}. The
    interpolating functions
    $V_\tau$ and $\bar V_\tau$
    satisfy the discrete equation
    \begin{equation}
      \label{eq:135}
      V'_\tau(t)\in -\partial_F^-\FF(\bar V_\tau(t))\quad\text{for
        a.e.~}t\in (0,T).
    \end{equation}
  \item
    There exists a constant $C$ independent of $\tau$
    such that for every discrete solution and for every $\tau\in (0,\tau_*)$
    \begin{align}\label{eq:4.33}
      \sup_{t\in [0,T]}|
      V_\tau(t)|
      \le
      \sup_{t\in [0,T]}|\bar V_\tau(t)|
      &\leq C\\
      \label{eq:4.34}
      \sup_{t\in [0,T]}
      \FF(\bar V_\tau(t))
      \leq \FF(V_0)&\le C,\\
        \label{eq:4.35}
      \|V_\tau'\|_{L^2(0,T;\Hilb)}
      &\leq C,\\
        \|\bar V_\tau-V_\tau\|_{L^\infty(0,T;\Hilb)}
      &\le C\tau^{1/2}.
    \end{align}
    \item There exists a constant $c$ independent of $\tau$ such that
  $V_\tau(t)\in \Kc c$ and $\bar V_\tau(t)\in \Kc c$ for every $t\in [0,T]$.
  \end{enumerate}
\end{lemma}
\begin{proof}
  \textbf{(1)}
  The existence of discrete solutions to the Minimizing Movement
  scheme follows directly from Lemma \ref{le:Flsc}.
  Notice that in our case we did not assume that the sublevels of
  $\FF(V)+\frac1{2\tau_*}|V|^2$ are strongly compact as in
  \cite[Lemma 1.2]{rs}; however Lemma \ref{le:Flsc}
  guarantees the weak lower
  semicontinuity of $\FF$ and the weak compactness of the sublevels
  of $\FF(V)+\frac1{2\tau_*}|V|^2$.

  \eqref{eq:135} is then a simple application of the definition of
  Fr\'echet subdifferential (see e.g.~\cite[(4.29)]{rs}).
 In fact the minimality of $V^n_\tau$ in \eqref{impleuler} yields
  \begin{align*}
    \FF(W)-\FF(V^n_\tau)
    &\ge
      \frac{1}{2\tau}|V^n_\tau-V^{n-1}_\tau|^2-
      \frac{1}{2\tau}|W-V^{n-1}_\tau|^2
      \\&=
      -\frac1\tau\langle V^n_\tau-V^{n-1}_\tau,W-V^{n}_\tau\rangle-
      \frac{1}{2\tau}|W-V^n_\tau|^2;
  \end{align*}
  using \eqref{eq:136}
  with $\omega(Z):= \frac{1}{2\tau}|Z|^2$ we get
  \begin{equation}
    \label{eq:138}
    -\frac{V^n_\tau-V^{n-1}_\tau}\tau\in \partial_F^-\FF(V^n_\tau).
  \end{equation}
  \eqref{eq:135} then follows 
  since for every $1\le n\le N(\tau)$ 
  \begin{equation}
    \label{eq:137}
    V_\tau'(t)=\frac{V^n_\tau-V^{n-1}_\tau}\tau\quad\text{in
    }(t_{n-1},t_n).
  \end{equation}
  \textbf{(2)}
  is a direct application of \cite[Prop.~4.6]{rs}.

  \noindent
  \textbf{(3)} still follows by \eqref{eq:4.33}, \eqref{eq:4.34},
  Lemma \ref{le:Flsc}, and the convexity of $\Kc c$.
\end{proof}
\begin{lemma}[Weak convergence of the Minimizing Movement scheme]
  \label{le:weak-convergence}
  Under the same assumptions of Lemma \ref{le:h1h2h3ok}
  from every vanishing sequence $k\mapsto \tau(k)\downarrow 0$
  it is possible to extract a further subsequence (not relabeled) and
  to find a limit function
  \begin{equation}
    \label{eq:60}
    V\in H^1(0,T;\Hilb),\quad
    \sup_{t\in [0,T]}\FF(V(t))\le \FF(V_0)
  \end{equation}
   such
  that
  \begin{align}
    \label{eq:139}
    \bar V_{\tau(k)}(t)\weakto V(t),\quad
    V_{\tau(k)}(t)\weakto V(t)\quad&\text{weakly in  }\Hilb\text{ for
                                     every }t\in [0,T],\\
    \label{eq:140}
    V_{\tau(k)}'\weakto V'\quad&\text{weakly in }L^2(0,T;\Hilb).
  \end{align}
\end{lemma}
\begin{proof}
  The proof of the weak convergence is a simple application of the a
  priori estimates of Lemma \ref{le:h1h2h3ok}
  (see also \cite[Prop.~2.2.3]{ags}, by choosing as $\sigma$
  the weak topology of $\Hilb$).
\end{proof}
\subsection{Strong convegence of
  the Minimizing Movements scheme}\label{sec:convergence}
This section contains the crucial argument improving Lemma
\ref{le:weak-convergence}, which is based on a compensated compactness strategy.
\begin{proposition}\label{primastima}
  Let $V_k:=V_{\tau(k)}$, $\bar V_k:=\bar V_{\tau(k)}$ be sequences of
  discrete solutions of the Minimizing Movement scheme weakly converging to
  $V$ along a decreasing sequence of step sizes $\tau(k)\downarrow0$
  as in Lemma
  \ref{le:weak-convergence}. 
  Then $V_k,\bar V_k\rightarrow V$ uniformly in $\Hilb$
  so that $V$ is a Generalized Minimizing Movement in $\mathrm{GMM}(\Phi,V_0,T)$.
\end{proposition}
\begin{proof}
First of all we note that $V_k,\bar V_k$ satisfy the differential
inclusion
\begin{equation}\label{diffinclappr}
 V_k'(t)\in -\partial ^-_F\FF(\bar V_k(t)) \quad \text{for a.e.~}t\in (0,T)
\end{equation}
as in \eqref{eq:135}, the apriori estimates of Lemma
\ref{le:h1h2h3ok}, and the weak convergences \eqref{eq:139} and \eqref{eq:140}.
By Lemma \ref{Elk1} we can decompose $-V_k'(t)$ as the sum of two
piecewise constant terms 
\begin{equation}
  \label{eq:119}
  -V_k'(t)=A_k(t)+B_k(t)
  -\theta \bar V_k(t),\quad
  A_k(t)\in \subH(\bar V_k(t)),\quad
  B_k(t)\in \partial_F^-\GG_\theta(\bar V_k(t)),
\end{equation}
where $S(\bar V_k(t))$ was defined in~\eqref{eq:93b} and~\eqref{eq:108}.
For the sake of clarity, we now divide the proof in several steps.

{\bf Step 1: compactness of $A_k$.}
Thanks to Proposition \ref{prop:compactness}
the image of $A_k$ is contained in a compact set $\mathcal C\subset
\Hilb$ independent of $k$.
For late use, we will introduce 
\[
  \mathcal C_{0,t}:=t\, \cconv(\mathcal C\cup \{0\})
  =\big\{tx:x\in \cconv(\mathcal C\cup \{0\})\big\},\quad t\ge 0.
\]
By~\cite[Theorem~3.25]{rudin}, we deduce that $\mathcal C_{0,t}$ is
a family of 
compact sets in $\Hilb$, which by definition are also a convex,
contain the origin, and satisfies $\mathcal C_{0,t}\subset \mathcal
C_{0,T}$
for every $t\in [0,T]$ since
$\mathcal C_{0,t}=\frac tT \mathcal C_{0,T}$ and  $\mathcal C_{0,T}$
is a convex set containing $0$.

As a consequence, $\int_0^T|A_k(t)|^2\,dt$ is uniformly bounded and, up to pass to subsequences, \[
A_k\rightharpoonup A,\qquad \text{weakly in $L^2(0,T;\Hilb)$ as $k\to+\infty$}.
\]

{\bf Step 2: a limsup inequality.}
At this point we only have that, for all $t\in[0,T]$, \[
  B_k
  \rightharpoonup B:=-V'-A+\theta V
  \quad \text{weakly in $L^2(0,T;\Hilb)$ as $k\to+\infty$}, 
\]
We want now to prove
\begin{equation}\label{firsteqprimastima}
  \limsup_{k\to\infty}\int_0^T
  \mathrm e^{-2\theta t}\langle B_k(t),V_k(t)\rangle\,\d t
  \leq \int_0^T \mathrm e^{-2\theta t}\langle B(t), V(t)\rangle\,\d t.
\end{equation}
We first introduce the perturbation $C_k:=B_k+\theta(V_k-\bar V_k)=
-V_k'-A_k+\theta V_k$ ;
since
$\sup_{t\in [0,T]}|\bar V_k(t)-V(t)|\to 0$ as $k\to\infty$,
\eqref{firsteqprimastima} is equivalent to
\begin{equation}\label{eq:f}
  \limsup_{k\to\infty}\int_0^T
  \mathrm e^{-2\theta t}\langle C_k(t),V_k(t)\rangle\,\d t
  \leq \int_0^T \mathrm e^{-2\theta t}\langle B(t), V(t)\rangle\,\d t.
\end{equation}
By definition of $C_k$ we have
\begin{align}
  \notag
  \int_0^T\mathrm e^{-2\theta t}\langle C_k(t),V_k(t)\rangle \,\d
  t&=-\int_0^T
  \mathrm e^{-2\theta t}\langle V_k'(t)
  -\theta  V_k,V_k(t) \rangle \,\d t
  -\int_0^T\mathrm e^{-2\theta t}
     \langle A_k(t),V_k(t) \rangle \,\d t\\
  \notag
  &=-\int_0^T
  \frac\d{\d t}\Big(\frac 12\mathrm e^{-2\theta t}|V_k(t)|^2\Big)\,\d t
  -\int_0^T\mathrm e^{-2\theta t}
    \langle A_k(t),V_k(t) \rangle \,\d t\\
  \label{eq:124}
  &=\frac12|V_0|^2-\frac12
  \mathrm e^{-2\theta T}|V_k(T)|^2-\int_0^T
  \mathrm e^{-2\theta t}\langle A_k(t),V_k(t) \rangle \,\d t,
\end{align}
and a similar calculation holds for $B$:
\begin{equation}
  \label{eq:125}
  \int_0^T\mathrm e^{-2\theta t}\langle B(t),V(t)\rangle \,\d
  t
  =\frac12|V_0|^2-\frac12
  \mathrm e^{-2\theta T}|V(T)|^2-\int_0^T
  \mathrm e^{-2\theta t}\langle A(t),V(t) \rangle \,\d t.
\end{equation}
The lower semicontinuity of the norm with respect to the weak convergence yields\[
\limsup_{k\to+\infty} -|V_k(T)|^2 \leq -|V(T)|^2.
\]
Hence
\eqref{firsteqprimastima}
will follow if we to prove the convergence
\begin{equation}
  \label{eq:120}
  \lim_{k\to\infty}\int_0^T
  \mathrm e^{-2\theta t}\langle A_k(t),V_k(t) \rangle \,\d t=
  \int_0^T
  \mathrm e^{-2\theta t}\langle A(t),V(t) \rangle \,\d t.
\end{equation}
We use a compensated-compactness argument and we introduce
the integral function
$$\mathcal A_k(t):=\int_0^t\mathrm e^{-2\theta s} A_k(s)\,\d s
\quad t\in [0,T].$$
Since for every $k\in \N$ 
the sequence $\mathcal A_k'(t)=\mathrm e^{-2
  \theta t} A_k(t)$
takes values
in the compact subset $\CC\subset \Hilb$ and thus is uniformly bounded,
we deduce that $\mathcal A_k$ is uniformly Lipschitz
equicontinuous.

It is also easy to show that $\mathcal A_k(t)\in \mathcal C_{0,T}$
for every $k\in \N$ and every $t\in [0,T]$, since
by Jensen inequality
\begin{displaymath}
  t^{-1}\mathcal A_k(t)=\mean{0}^{t}\mathrm e^{-2\theta t}
  A_k(t)\,\d t\in \cconv\big(\mathcal C\cup \{0\}\big)
  \quad\text{for every }t\in (0,T].
\end{displaymath}
All in all, by Ascoli-Arzel\`a Theorem, we deduce that $(\mathcal
A_k)_k$ is relatively compact in $C^0([0,T]; \Hilb)$, and therefore 
\begin{equation}\label{eq:strongL2A}
  \mathcal A_k\rightarrow \mathcal A
  \text{ uniformly and in $L^2(0,T;\Hilb)$ as $k\to+\infty$},
\end{equation}
where $\displaystyle \mathcal A(t):=
  \int_0^t\mathrm e^{-2\theta s} A(s)\,\d s$.
  An integration by parts then gives
  \begin{equation}
  \int_0^T\mathrm e^{-2\theta t}\langle A_k(t),V_k(t) \rangle \,\d t
  =-\int_0^T\langle \mathcal A_k(t),V_k'(t)\rangle \,\d t+
  \langle \mathcal A_k(T), V_k(T)\rangle-\langle \mathcal A_k(0), V_k(0)\rangle
\label{eq:121}
\end{equation}
with a similar identity involving $A,V$, and $\mathcal A$.
Then we combine~\eqref{eq:strongL2A},~\eqref{eq:139}
and \eqref{eq:140} to infer
\[
\langle \mathcal A_k(T), V_k(T)\rangle\to \langle \mathcal A(T), V(T)\rangle,\qquad\langle \mathcal A_k(0), V_k(0)\rangle\to\langle \mathcal A(0), V(0)\rangle
\]
and
\[
\lim_{k\to+\infty}\int_0^T\langle \mathcal A_k(t),V_k'(t)
\rangle\,dt=\int_0^T\langle\mathcal A(t),V'(t) \rangle\,\d t;
\]
we can then pass to the limit in \eqref{eq:121}
and we get \eqref{eq:120}.

{\bf Step 3: for a.e. $t\in (0,T)$ we have $B(t)\in
  \partial^-\GG_\theta(V(t))$.}
Introducing the integral functional
\begin{equation}
  \label{eq:122}
  \widetilde{\mathscr K}_\theta(V):=
  \int_0^T \mathrm e^{-2\theta t}\GG_\theta(V)\,\d t
\end{equation}
in the Hilbert space $\widetilde\Hilb:=L^2((0,T),\mu_\theta,\Hilb)$
associated with the Borel measure $\mu_\theta:=\mathrm e^{-2\theta
  t}\mathcal L^1$ in $(0,T)$,
since $V\in D(\widetilde\GG_\theta)$ and $B\in \widetilde\Hilb$,
we can equivalently prove that
for all $W\in D(\widetilde\GG_\theta)\subset \widetilde\Hilb$
\begin{equation}
  \int_0^T\mathrm e^{-2\theta t}\langle B(t),W(t)-V(t)\rangle\,\d t
  \leq \int_0^T \mathrm e^{-2\theta
    t}\GG_\theta(W(t))\,dt-\int_0^T\mathrm e^{-2\theta t}\GG_\theta(V(t))\,\d
  t.\label{eq:123}
\end{equation}
Since $B_k(t)\in \partial^-\GG_\theta(\bar V_k(t))$
we have 
\[
  \int_0^T\mathrm e^{-2\theta t}\langle B_k(t),W(t)-\bar V_k(t)\rangle\,\d t
  \leq
  \int_0^T \mathrm e^{-2\theta t}\GG_\theta(W(t))\,\d t-\int_0^T\GG_\theta(\bar
  V_k(t))\,\d t.
\]
Then it is sufficient to use Step~2, the weak lower semicontinuity of
$\widetilde\GG_\theta$ in $\widetilde\Hilb$
(since it is strongly lower semicontinuous and $(-\theta)$-convex),
the weak convergence of $B_k$, and the strong convergence of
$\bar V_k-V_k$ to $0$ in $\widetilde\Hilb$
to obtain \[
\begin{split}
  \int_0^T
  \mathrm e^{-2\theta t}\langle B(t),W(t)-V(t)\rangle\,\d t
  &\leq \liminf_{k\to\infty}
  \int_0^T
  \mathrm e^{-2\theta t}\langle B_k(t),W(t)-V_k(t)\rangle\,\d t\\
  & =\liminf_{k\to\infty}
  \int_0^T \mathrm e^{-2\theta t}\langle B_k(t),W(t)-\bar V_k(t)\rangle\,\d t\\
  &\leq \int_0^T \mathrm e^{-2\theta t}\GG_\theta(W(t))\,\d t+
  \liminf_{k\to\infty}\Big(-\int_0^T \mathrm e^{-2\theta t}
  \GG_\theta(\bar V_k(t))\,\d t\Big)\\
  &\leq \int_0^T \mathrm e^{-2\theta t}\GG_\theta(W(t))\,\d t-\int_0^T
  \mathrm e^{-2\theta t}\GG_\theta(V(t))\,\d t,
\end{split}
\]
which means $B(t)\in \partial^-\GG_\theta(V(t))$ for a.e. $t$.

{\bf Step 4: $V_k\rightarrow V$ uniformly in $\rmC^0([0,T];\Hilb)$.}
By the equicontinuity estimate
and the weak convergence \eqref{eq:140},
it is sufficient to prove that for all $S\in (0,T]$ 
\begin{equation}\label{convatT}
  \limsup_{k\rightarrow +\infty}|V_k(S)|^2=|V(S)|^2.
\end{equation}
Using the identities \eqref{eq:124} and \eqref{eq:125}
written in the interval $[0,S]$
and taking into account of \eqref{eq:120},
\eqref{convatT} is equivalent to
\begin{equation}\label{eq:liminf1}
  \limsup_{k\to\infty}-
  \int_{0}^S\mathrm e^{-2\theta t}\langle C_k(t),V_k(t)\rangle\,\d t
  \le -\int_0^S\langle \mathrm e^{-2\theta t}B(t), V(t)\rangle\,\d t.
\end{equation}
Recalling that $C_k=B_k+\theta(V_k-\bar V_k)$
and $V_k-\bar V_k\to 0$ uniformly in $\Hilb$,
\eqref{eq:liminf1} can be reduced to
\begin{equation}\label{eq:liminf}
  \liminf_{k\to\infty}\int_{0}^T
  \mathrm e^{-2\theta t}\langle B_k(t),\bar V_k(t)\rangle\,\d t
  \ge \int_0^T \mathrm e^{-2\theta t}\langle B(t), V(t)\rangle\,\d t.
\end{equation}
Since by step 3 we have $B(t)\in \GG_\theta(V(t))$ a.e., the
monotonicity property of the subdifferential of a convex function
yields
\begin{equation}
  \langle B_k(t)-B(t),\bar V_k(t)-V(t)\rangle \ge 
  0\quad\text{for a.e.~$t\in (0,T)$}\label{eq:126}
\end{equation}
so that
\[
\begin{split}
  \langle B_k(t),\bar V_k(t)\rangle-\langle B(t),V(t)\rangle
  &=\langle B_k(t)-B(t), \bar V_k(t)-V(t)\rangle+
  \langle B(t),\bar V_k(t)\rangle-\langle B_k(t),V(t)\rangle
  \\&\geq \langle B(t),\bar V_k(t)\rangle
  -\langle B_k(t),V(t)\rangle
\end{split}
\]
and therefore
\[
\begin{split}
  \liminf_{k\to\infty}\int_0^T&\mathrm e^{-2\theta t}\Big(\langle
B_k(t),V_k(t)\rangle-\langle B(t),V(t)\rangle\Big)\,\d t\\
&\geq \liminf_{k\to\infty}
\int_{0}^T \mathrm e^{-2\theta t}\langle B(t),V_k(t)\rangle\,\d t-
\int_{0}^T \mathrm e^{-2\theta t}\langle
B_k(t),V(t)\rangle\,\d t=0
\end{split}
\]
where we used the weak convergence of $B_k$ to $B$ and of $V_k$ to
$V$.
\end{proof}

\subsection{Proof of the main results of Section \ref{sec:main}}
\label{subsec:proofs}
We can now collect all the information on the convergence of the
Minimizing Movement scheme to conclude the proofs of the main results
of
Section \ref{sec:main}.

\subsubsection*{\bfseries Proof of Theorem \ref{thm:main1}}
The fact that $\mathrm{GMM}(\Phi,V_0,T)$ is not empty
just follows from Proposition \ref{primastima}.
We can then apply \cite[Theorem 3]{rs} which shows that
every element $V\in \mathrm{GMM}(\Phi,V_0,T)$ satisfies \eqref{eq:127ter}, 
\eqref{eq:127bis}, \eqref{eq:128} and \eqref{eq:130}
if $\FF$ satisfies the Chain rule property we proved in Proposition
\ref{prop:chain}.
In fact the compactness assumption in
\cite[Theorem 3]{rs} was just needed to guarantee the existence of an
element in $\mathrm{GMM}(\Phi,V_0,T)$ but the proof of the
characterization of the limiting subdifferential equation is independent of
such an assumption.

\subsubsection*{\bfseries Proof of Proposition \ref{prop:better}}
It is sufficent to combine Theorem \ref{thm:main1} with Corollary
\ref{cor:vitafacile}.

\subsubsection*{\bfseries Proof of Theorem \ref{thm:main2}}
It just follows by Theorem \ref{thm:main1} and \eqref{eq:59bis} of Lemma \ref{ELk}.

\appendix

\section{Convergence of eigenvalues and eigenfunctions for
  Schr\"odinger potentials}\label{sec:appmosco}

In order to study the behaviour of the eigenvalues of $\E_V$ with
respect to $V$
we will use Mosco convergence in $\Hilb$.
Recall that a sequence
of functionals $\Phi_n:\Hilb\to \R\cup\{+\infty\}$ converges in the sense of Mosco to a limit functional $\Phi:\Hilb\to\R\cup\{+\infty\}$ if
the following two conditions hold:
\begin{enumerate}[(M1)]
\item for every sequence $(w_n)_{n\in \N}\subset  \Hilb$ weakly converging to $w\in
  \Hilb$ we have $\liminf_{n\to\infty}\Phi_n(w_n)\ge \Phi(w)$;
\item for every $v\in \Hilb$ there exists a sequence $(w_n)_{n\in \N}$
  strongly converging to $w$ such that
  $\lim_{n\to\infty}\Phi_n(w_n)=\Phi(w)$.
\end{enumerate}
Mosco convergence is equivalent to $\Gamma$-convergence with respect
to the weak and strong $\Hilb$-topology, see
\cite[Chapters 12,13]{DalMaso}.
Under equi-coercivity (guaranteed in our case
by the compactness of the imbedding of $\V$ in $\Hilb$),
weak and strong $\Gamma$-convergence are equivalent and are also
related to uniform convergence of the resolvents.

We split the proof of Lemma~\ref{le:conteigpot} in two parts: first we prove that the weak convergence of potentials implies the Mosco convergence of the associated functionals, and then show that the Mosco convergence implies the convergence of eigenvalues and eigenfunctions.
\begin{lemma}\label{le:conteigapp}
Let $V_n\in \K$, $n\in \N$, be a sequence weakly converging in
$\Hilb$ to $V\in\K$ as $n\to+\infty$.
Then the corresponding sequence of quadratic forms $\E_{V_n}$
converges
in the sense of Mosco to $\E_V$.
\end{lemma}
\begin{proof}
We start from the condition (M1) and consider a sequence
$w_n$ weakly converging to $w$ in $\Hilb$ such that $\E_{V_n}(w_n)\le
E$ definitely. In particular $\E(w_n)$ is uniformly bounded from
above, so that $w_n$ is converging strongly to $w$ in $\Hilb$ and
\begin{equation}
  \label{eq:16}
  \liminf_{n\to+\infty}\E(w_n)+V_{\rm min} \int_\sfD |w_n|^2\,\d\mm\ge
  \E(w)+V_{\rm min}\int_\sfD |w|^2\,\d\mm.\end{equation}
On the other hand, for every $k>0$, $w_n\land k$ converges strongly to
$w\land k$ in $L^4(\sfD,\mm)$ so that 
\begin{align*}
  \liminf_{n\to+\infty}\int_\sfD (V_n-V_{\rm min})|w_n|^2\,\d\mm
  \ge \liminf_{n\to+\infty}\int_\sfD (V_n-V_{\rm min})|w_n\land k|^2\,\d\mm=
  \int_\sfD (V-V_{\rm min})|w\land k|^2\,\d\mm.
\end{align*}
Since $k>0$ is arbitrary we conclude that
\begin{equation}
  \label{eq:14}
  \liminf_{n\to+\infty}\int_\sfD (V_n-V_{\rm min})|w_n|^2\,\d\mm\ge
  \int_\sfD (V-V_{\rm min})|w|^2\,\d\mm.
\end{equation}
Combining \eqref{eq:16} and \eqref{eq:14} we obtain $\E_V(w)\le E$ as
well.

Concerning (M2), we first show that 
for every $w\in D(\E_V)$
there exists a sequence $(w_k)_{k\in \N}$ in $D(\E_V)\cap
L^4(\sfD,\mm)$ converging strongly to $w$ in $\Hilb$ such that
$\E_V(w_k)\to \E_V(w)$ as $k\to+\infty$.
It is sufficient to consider an orthonormal basis of eigenfunctions
$(u_h)_{h\in \N}$ for $\E_V$ and set
\begin{equation}
  \label{eq:17}
  w_k:=\sum_{h=1}^k \langle w,u_h\rangle u_h
\end{equation}
On the other hand,
for every $k\in \N$ we have
\begin{equation}
  \label{eq:18}
  \lim_{n\to+\infty}\E_{V_n}(w_k)=\E_V(w_k)
\end{equation}
so that a standard diagonal argument yields (M2).
\end{proof}

Now we provide the proof of some well-known facts concerning the Mosco convergence and the convergence of eigenvalues.

\begin{definition}\label{def:resolvent}
  For all $\beta>\lambda_{\rm min}$ and $V\in \K$,
  the resolvent operator $\rmR_V^\beta \colon \Hilb\to \Hilb$
  maps every
  $f\in \Hilb$ into the the unique solution $u$ of the problem \[
    \E(u,w)+\int_{\sfD}(V+\beta)u\,w\,\d\mm=\int_{\sfD}f w\,\d\mm
    \qquad \text{for all }w\in D(\E_V).
\]
$\rmR_V^\beta f$ is the unique minimizer of the functional \[
  v\mapsto \frac 12\E(v)+\frac 12
  \int_{\sfD}(V+\beta)v^2\,\d\mm-\int_\sfD fv\,\d\mm
\]
\end{definition}
We list here some properties of the resolvent operator:
\begin{itemize}
\item The operator $\rmR_{V}^\beta$ is continuous.
\item The operator $\rmR_{V}^\beta $ is compact, thanks to the compact embedding of $\V$ into $\Hilb$.
\item The operator $\rmR_{V}^\beta $ is self-adjoint.
\item The operator $\rmR_{V}^\beta $ is positive.
\end{itemize}
As a consequence, the spectrum of $\rmR_{V}^\beta $ is real, positive and discrete and it is made of eigenvalues ordered as \[
  0\leq \dots \leq \Lambda_k(\beta,V)\leq \dots \leq \Lambda_1(\beta,V)=
  \|R_{V}^\beta \|_{\mathcal L(\Hilb)},
\]
which are related to the sequence $\lambda_k(V)$ by the formula
\begin{equation}
  \label{eq:144}
  \Lambda_k(V)=(\lambda_k(V)+\beta)^{-1},\quad
  \text{$u$ is a $(V,\lambda_k(V))$-eigenfunction if and only if }
  \rmR_V^\beta u=(\lambda_k(V)+\beta)^{-1}u.
\end{equation}
The next fundamental lemma relates Mosco convergence
to the (uniform) norm convergence of the resolvent operators.
\begin{lemma}\label{le:A1}
  Let $V_n,V\in \K$ and let us assume that $V_n\weakto V$ in $\Hilb.$
  Then for every $\beta>\lambda_{\rm min}$
  the associated resolvent operators converge, namely \[
    \rmR_{V_n}^\beta \rightarrow \rmR_V^\beta
    \qquad \text{in }\mathcal L(\Hilb)
    \quad\text{as $n\to +\infty$.}
\]
\end{lemma}
\begin{proof}
  We fix $\beta>\lambda_{\rm min}$.
  From the definition of operator convergence, for $t\geq 0$ fixed, we have \[
\|\rmR_{V_n}^\beta -\rmR_V^\beta \|_{\mathcal L(\Hilb)}=\sup_{|f|\leq 1}|\rmR_{V_n}^\beta (f)-\rmR_V^\beta (f)|\leq |\rmR_{V_n}^\beta (f_n)-\rmR_V^\beta (f_n)|+\frac1n,
\]
for a suitable
sequence $(f_n)\subset \Hilb$ with $|f_n|\leq 1$ and that we can assume to be weakly-$\Hilb$ converging to some $f\in \Hilb$.
We can then split
\[
  |\rmR_{V_n}^\beta (f_n)-\rmR_V^\beta (f_n)|\leq
  |\rmR_{V_n}^\beta (f_n-f)|+
  |\rmR^\beta_{V_n}(f)-\rmR_V^\beta (f)|+|\rmR_V^\beta (f_n)-\rmR_V^\beta (f)|.
\]
The last term is vanishing, as the resolvent operator is
continuous and $\|\rmR_V^\beta (f_n)\|_{\V}$ is uniformly bounded.

The second term is also infinitesimal thanks to
\cite[Theorem 13.12]{DalMaso}.
Concerning the first term, since
$\rmR_{V_n}^\beta (f_n-f) $ is uniformly bounded in $\V$, it is
sufficient to prove its weak convergence in $\Hilb$.
For every $g\in \Hilb$ we have
\begin{displaymath}
  \langle \rmR_{V_n}^\beta (f_n-f),g\rangle=
  \langle f_n-f, \rmR_{V_n}^\beta g\rangle\to 0
\end{displaymath}
since $\rmR_{V_n}^\beta g\to \rmR_V^\beta g$ strongly in $\Hilb$
and $f_n\weakto 0$.
In conclusion, $w=u=\rmR_V^\beta (f)$ and by the compact embedding of $\V$ into $\Hilb$, we conclude that \[
\rmR_{V_n}^\beta (f_n)\rightarrow \rmR_V^\beta (f),\qquad \text{strongly in }\Hilb,
\]
and the convergence holds for the whole sequence, since the limit is independent of the chosen subsequence.
\end{proof}

Eventually, thanks to the classical theory of linear operators, the
norm convergence of the operators implies the convergence of the
spectrum, see for example~\cite[Lemma XI.9.5]{dunford}.
Passing to the limit in the equation
\begin{displaymath}
  \rmR_{V_n}^\beta u_n=\Lambda_n u_n
\end{displaymath}
where $u_n$ is normalized sequence of eigenvalues
associated with a converging sequence $\Lambda_n$ 
and using the uniform boundedness of $\rmR_{V_n}^\beta$
and the compactness of the embedding of $\V$ in $\Hilb$
we can also prove the convergence (possibly up to subsequences)
of the eigenfunctions: this concludes the proof of Lemma~\ref{le:conteigpot}.

\section{Trace of symmetric operators}

Let $E\subset \Hilb$ be a subspace. 
We denote by $\Ort^k(E)$ the subset of orthonormal sets of $E^k$:
\begin{equation}
  \label{eq:27b}
  \Ort^k(E):=\Big\{\ww=(w_1,\cdots,w_k)\in E^k:\langle w_i,w_j\rangle =\delta_{ij}\Big\}.
\end{equation}
If $E$ has finite dimension $\dim(E)=d$ then we set $\Ort(E):=\Ort^d(E)$.
An orthogonal matrix $\sfQ\in \mathrm O(k)$ operates on $E^k$ 
by
\begin{equation}
  \label{eq:32}
  (\sfQ\ww)_j:=\sum_{i=1}^k \sfQ_{ij}w_j,\quad
  \ww=(w_1,\cdots,w_k)\in \Hilb^k,\quad j=1,\cdots,k.
\end{equation}
It is clear that if $\ww\in \Ort^k(E)$ then
also $\sfQ\ww\in \Ort^k(E)$ since
\begin{displaymath}
  \langle (\sfQ\ww)_i,(\sfQ\ww)_j\rangle
  =
  \langle \sum_{h=1}^k\sfQ_{hi}w_h,\sum_{l=1}^k\sfQ_{lj}w_l\rangle
  =
  \sum_{h,l=1}^k\sfQ_{hi}\sfQ_{lj}\langle w_h,w_l\rangle
  =
  \sum_{h,l=1}^k\sfQ_{hi}\sfQ_{lj}\delta_{hl}=
\sum_{h=1}^k\sfQ_{hi}\sfQ_{hj}=\delta_{ij}
\end{displaymath}
Let now $\calQ$ be a symmetric bilinear form on $E$
and let
$\sfQ\in \mathrm O(k)$ be an orthogonal matrix.
For every $\ww\in \Ort^k(E)$ with $\ww'=\sfQ\ww$ we have
\begin{equation}
  \label{eq:33}
  \sum_{h=1}^k \calQ(w_h,w_h)=
  \sum_{h=1}^k \calQ(w'_h,w_h').
\end{equation}
In fact
\begin{align*}
  \sum_{h=1}^k \calQ(w'_h,w_h')
  &=
    \sum_h\calQ\Big(\sum_{i}\sfQ_{ih}w_i, \sum_{j}\sfQ_{jh}w_j\Big)
    =\sum_h\sum_{i,j}\sfQ_{ih}\sfQ_{jh}\calQ(w_i, w_j)=
  \\&=\sum_{i,j}\calQ(w_i, w_j)\Big(\sum_h \sfQ_{ih}\sfQ_{jh}\Big)
  =\sum_{i,j}\calQ(w_i, w_j)\delta_{ij}=
  \sum_{h=1}^k \calQ(w_h,w_h).
\end{align*}
In particular, if $E$ has finite dimension $\dim(E)=d$ the quantity
\begin{equation}
  \label{eq:31}
  \mathrm{tr}_E(\calQ):=\sum_{h=1}^d\calQ(w_h,w_h),\quad\ww\in \Ort(E)
\end{equation}
is well defined and independent of the choice of $\ww\in \Ort(E)$.

A first application concerns the function 
\begin{equation}
  \label{eq:28}
  |\ww|^2(x):=\sum_{h=1}^d |w_h(x)|^2\quad x\in \sfD,\quad \ww\in \Ort(E)
\end{equation}
which is defined $\mm$-a.e.~in $\sfD$ and defines a quadratic form on $E$.
\begin{corollary}
  \label{cor:a1}
  If $E\subset \Hilb$ is a finite dimensional space with $\dim(E)=d$
  and
  $\ww',\ww''\in \Ort(E)$ then
  $|\ww'|^2(x)=|\ww''|^2(x)$ for $\mm$-a.e.~$x\in \sfD$.
\end{corollary}
\begin{proof}
  Since $\ww',\ww''$ are orthonormal basis of $E$ there exists an
  orthogonal matrix $\mathsf Q\in \mathrm O(d)$ such that
  $\ww''=\sfQ\ww'$. It is then sufficient to apply \eqref{eq:33}.
\end{proof}

\section{Basic facts concerning non smooth differential calculus}
\label{sec:basic}
Let
$C$ be a compact metrizable topological space, let $f:C\to \Hilb$ be a continuous map with image
$R:=f(C)$,
and let $g:C\to \R$
be a lower semicontinuous map.
We denote by $\mathscr K(R)$ the space of compact
subsets of $R$.
We set
\begin{equation}
  \label{eq:87}
  F(v):=\min\Big\{\langle v,f(u)\rangle+g(u),\quad
  u\in C\Big\}\quad
  \text{for every }v\in \Hilb,
\end{equation}
and we denote by $M(v)$ the set of $u\in C$ where the minimum
in \eqref{eq:87} is attained.
\begin{lemma}
  \label{le:a1}
  $F$ is a Lipschitz concave function whose superdifferential is given by
  \begin{equation}
    \label{eq:88}
    \partial^+ F(v)=\cconv \big(f(M(v)));
  \end{equation}
  in particular, for every $\xi\in \partial^+ F(v)$ there exists a
  Borel probability measure $\mu\in \mathscr P(C)$ such that
  \begin{equation}
    \label{eq:91}
    \supp(\mu)\subset M(v),\quad
    \xi=\int_C f(u)\,\d\mu(u).
  \end{equation}
  The map $\partial^+F:\Hilb\to \mathscr K(R)$ is weakly-strongly upper
  semicontinuous and satisfies
  \begin{equation}
    \label{eq:90}
    v_n\rightharpoonup v,\quad
    \xi_n\in \partial^+ F(v_n)\quad
    \Rightarrow\quad
    \left\{
      \begin{aligned}
        &(\xi_n)_{n\in \N}\text{ is strongly relatively compact in }\Hilb,\\
        &\text{every limit point $\xi$ of $(\xi_n)_{n\in \N}$ belongs to }\partial^+ F(v).
      \end{aligned}
    \right.
  \end{equation}
  $F$ is Fr\'echet differentiable at $v_0$ if and only if
  $f(M(v_0))$ is a singleton.
\end{lemma}
\begin{proof}
If $\xi=f(u)$ for some $u\in M(v)$ we have
\begin{align*}
  F(w)-F(v)\le \langle w,f(u)\rangle- g(u)-\big(\langle
  v,f(u)\rangle-g(u)\big)
  =\langle w-v,f(u)\rangle=
  \langle w-v,\xi\rangle
\end{align*}
showing that $\xi\in \partial^+F(v)$. It follows that
$f(M(v))\subset \partial^+F(v)$ and therefore also
$\cconv\big(f(M(v))\big)\subset \partial^+F(v)$.

Let us now prove that if $\xi\not\in \cconv\big(f(M(v))\big)$ then
$\xi\not\in \partial^+F(v)$. Since $\cconv\big(f(M(v))\big)$ is a
compact convex set, we can apply the second geometric form of
Hahn-Banach theorem and find $\eta\in \Hilb$ with $|\eta|=1$ and
$\alpha\in \mathbb R$ such that
\begin{displaymath}
  \langle \eta,\xi\rangle<\alpha< \min_{u\in M(v)}\langle \eta,f(u)\rangle,
\end{displaymath}
i.e.~the compact set $f(M(v))$ is contained in the open set
$H(\eta,\alpha):= \big\{x\in \Hilb:\langle
\eta,x\rangle>\alpha\big\}$.  By a standard compactness argument, we
can find $\varepsilon>0$ such that for every
$w\in \mathrm B(v,2\varepsilon)$ $f(M(w))\subset H(\eta,\alpha)$.
Choosing $w:=v+\varepsilon \eta$ and $u\in M(w)$ so that
$f(u)\in H(\eta,\alpha)$ we get
\begin{align*}
  \langle w-v,\xi\rangle
  &=\varepsilon \langle \eta,\xi\rangle
    <\eps\alpha<
    \eps \langle \eta,f(u)\rangle
    =\langle w-v,f(u)\rangle
  \\&=\langle w,f(u)\rangle-g(u)-
  \big(\langle v,f(u)\rangle-g(u)\big)
  =F(w)-\big(\langle v,f(u)\rangle-g(u)\big)
  \le F(w)-F(v)
\end{align*}
which shows that $\xi\not\in \partial^+F(v)$.

The representation \eqref{eq:91} is an immediate consequence of the
continuity of $f$ and the Krein-Milman Theorem.

Let now suppose that $v_n\weakto v$ in $\Hilb$ and let $\xi_n\in \partial^+F(v_n)$;
we can find a Borel probability measure $\mu_n$ on $C$ such that 
\begin{displaymath}
  \supp(\mu_n)\subset M(v_n),
  \quad
  \xi_n=\int_C f(u)\,\d\mu_n(u).
\end{displaymath}
Since $C$ is compact and metrizable, we can find a subsequence
$k\mapsto n(k)$ and a limit measure $\mu$ such that
$\mu_{n(k)}\to\mu$ weakly in $\mathscr P(C)$.
For every point $u$ of the support of $\mu$ there exists a sequence of
points
$u_n\in \supp(\mu_n)\subset M(v_n)$ converging to $u$; passing to the
limit in the family of inequalities
\begin{displaymath}
  \langle v_n,f(u_n)\rangle+g(u_n)\le \langle
  v_n,f(w)\rangle+g(w)\quad\text{for every }w\in C,
\end{displaymath}
we get
\begin{displaymath}
  \langle v,f(u)\rangle+g(u)\le \langle
  v,f(w)\rangle+g(w)\quad\text{for every }w\in C,
\end{displaymath}
so that $u\in M(v)$. It follows that setting
\begin{displaymath}
  \xi:=\int_C f(u)\,\d\mu(u)\in \partial^+ F(v)
\end{displaymath}
we then conclude that $\xi_{n(k)}\to \xi$ strongly in $\Hilb$ as
$k\to\infty$.

Concerning the Fr\'echet differential of $F$, it is obvious that
if $F$ is differentiable at $v_0$ then $\partial^+F(v_0)$ reduces
to a singleton. To prove the converse property,
let $\xi_0$ be the unique element of $f(M(v_0))$:
we have just to show that $\xi_0\in \partial^-F(v_0)$.
By \eqref{eq:90}, for every $\eps>0$ we can find $\delta>0$ such that
\begin{displaymath}
  f(M(w))\subset \mathrm B(\xi_0,\eps)\quad\text{for every }w\in
  \mathrm B(v_0,\delta).
\end{displaymath}
For every $w\in \rmB(v_0,\delta)$ and $\xi\in f(M(w))$ we thus have
\begin{displaymath}
  F(w)-F(v)-
  \langle \xi_0,w-v\rangle\ge \langle \xi-\xi_0,w-v\rangle\ge
  -|\xi-\xi_0|\cdot |w-v|\ge -\eps|w-v|
\end{displaymath}
which shows that $\xi_0\in \partial^-F(v_0)$.
\end{proof}



\end{document}